\documentclass[11pt]{article}

\usepackage{amsmath,amsthm,amsfonts,latexsym}

\begin{document}

\newtheorem*{theo}{Theorem}
\newtheorem*{pro} {Proposition}
\newtheorem*{cor} {Corollary}
\newtheorem*{lem} {Lemma}
\newtheorem{theorem}{Theorem}[section]
\newtheorem{corollary}[theorem]{Corollary}
\newtheorem{lemma}[theorem]{Lemma}
\newtheorem{proposition}[theorem]{Proposition}
\newtheorem{conjecture}[theorem]{Conjecture}

\theoremstyle{definition}
 \newtheorem{definition}[theorem]{Definition} 
  \newtheorem{example}[theorem]{Example}
   \newtheorem{remark}[theorem]{Remark}
   
\newcommand{\Naturali}{{\mathbb{N}}}
\newcommand{\Reali}{{\mathbb{R}}}
\newcommand{\Complessi}{{\mathbb{C}}}
\newcommand{\Toro}{{\mathbb{T}}}
\newcommand{\Relativi}{{\mathbb{Z}}}
\newcommand{\HH}{\mathfrak H}
\newcommand{\KK}{\mathfrak K}
\newcommand{\LL}{\mathfrak L}
\newcommand{\as}{\ast_{\sigma}}
\newcommand{\tn}{\vert\hspace{-.3mm}\vert\hspace{-.3mm}\vert}
\def\A{{\cal A}}
\def\B{{\cal B}}
\def\E{{\cal E}}
\def\F{{\cal F}}
\def\H{{\cal H}}
\def\K{{\cal K}}
\def\L{{\cal L}}
\def\N{{\cal N}}
\def\M{{\cal M}}
\def\gM{{\frak M}}
\def\O{{\cal O}}
\def\P{{\cal P}}
\def\S{{\cal S}}
\def\T{{\cal T}}
\def\U{{\cal U}}
\def\V{{\mathcal V}}
\def\qed{\hfill$\square$}

\title{Fourier series and twisted 
C*-crossed products}

\author{Erik B\'edos$^*$, Roberto Conti\\}
\date{\today}
\maketitle
\markboth{R. Conti, Erik B\'edos}{
}
\renewcommand{\sectionmark}[1]{}
\begin{abstract}
This paper is an invitation to Fourier analysis in the context of reduced twisted  C*-crossed products associated with discrete unital twisted C*-dynamical systems. We discuss norm-convergence of Fourier series, multipliers and summation processes. Our study relies in an essential way on the (covariant and equivariant) representation theory  of  C$^*$-dynamical systems on Hilbert C*-modules. It also yields some information on the ideal structure of  reduced twisted  C*-crossed products.   


\end{abstract}

\vfill
\thanks{\noindent $^*$  partially supported by the Norwegian Research
Council.\par

}

\newpage


\section{Introduction} 

Since its birth 
about two centuries ago, 
the  theory of Fourier series has been applied to a seemingly endless number 
of different situations and, accordingly, 
it has been the subject of intensive studies, 
especially in relation to various kinds of convergence 
and summation techniques. 
Among many others, 
the problem of determining conditions under which the Fourier series 
of a continuous periodic function on the real line is uniformly 
convergent has received a good deal of attention in the literature,  
and various kinds of summation processes  have also been constructed.

In the theory of operator algebras, started in the seminal work of D. Murray and J. von Neumann,
it is well known that one may associate to any group
several interesting examples of C$^*$-algebras and von Neumann algebras. 
In the context of twisted group C$^*$-algebras (and von Neumann algebras) associated with discrete groups,
the Fourier series of any element makes perfect sense. In the C$^*$-algebraic case,
the study of norm-convergence and summation processes is more involved than in the classical set-up,
but a surprisingly detailed analysis is possible, as exposed for instance in our previous article \cite{BeCo2}.
Now, given a twisted action $(\alpha, \sigma)$ of a discrete group $G$ on a unital C$^*$-algebra $A$
(the case we discuss in this paper),
one may also consider the Fourier series of any
element in the so-called (reduced) crossed product C$^*$-algebra $C_r^*(\Sigma)$, where 
$\Sigma$ denotes the quadruple $(A, G, \alpha, \sigma)$.  
However,  the Fourier coefficients lie now in $A$ (rather than in $\Complessi$), and consequently, the analysis becomes much more challenging.  
Our main aim with the present work is to investigate 
how one can handle this situation, having in mind the wild variety of different cases that may appear.
Due to the success of the theory of classical Fourier series,
we expect that 
once put on solid grounds 
the corresponding theory will become a useful tool 
in the study of C$^*$-dynamical systems.

We stress that the idea of doing Fourier analysis in reduced  C$^*$-crossed products is not new. 
It is already present in G.\ Zeller-Meier's impressive article \cite{ZM} from 1968, where he, among many other results, shows the existence of summation processes in the case of amenable groups, and uses this to obtain some valuable information about  the ideal structure. In the book by K.\ Davidson \cite{Da} one can find a proof of the direct analogue of the Fej\'er summation process in the special case of  
crossed products by an action of $\Relativi$ (see also \cite{Tom1}).  One may also consider $C^*_r(\Sigma)$ as the reduced cross sectional algebra of a  Fell bundle over the discrete group $G$ (see \cite{ExLa})
and notice that R. Exel  \cite{Ex} has shown how to associate Fourier series to elements in such algebras. In the same paper, Exel constructs summation processes in the case of Fell bundles with the so-called approximation property and illustrates their usefulness when studying induced ideals in the sectional algebra (see also \cite{Ex2}).    One could therefore think that one may as well work in the more general setting of Fell bundles. However, as alluded to in \cite{BeCo3}, our attitude has been that it should be possible to develop a more powerful analysis by exploiting the structure of discrete twisted 
C$^*$-crossed products and their representation theory on Hilbert  C$^*$-modules. 
We will do our best to justify this point of view and add some further evidence to the fact that the equivariant representations  of $\Sigma$ on Hilbert $A$-modules introduced in \cite{BeCo3} play a role complementary to the one played by covariant representations. Note that when $A$ is trivial, this splitting is not visible: covariant and equivariant representations coincide in this case and amount to unitary representations of $G$.
 
\smallskip
The starting point for our approach is as follows.  As is well known (see e.g. \cite{ZM, BeCo3}), $B=C^*_r(\Sigma)$ may be characterized (up to isomorphism) as a C$^*$-algebra $B$ that is generated by a copy of $A$ and  a family of unitaries
$\{u_g\}_{g \in G}$ satisfying the relations $u_g\,a = \alpha_g(a)\,u_g$ and $u_g\,u_h = \sigma(g,h)\,u_{gh}$, and is equipped with a faithful conditional expectation $E$ from $B$ onto $A$ satisfying $E(u_g) = 0$ when $g\neq e$ (the identity of $G$). The expectation $E$ may be thought as some kind of $A$-valued Haar integral: 
if $G$ is abelian and $\alpha, \sigma$ are both trivial, then $C^*_r(\Sigma)$ is isomorphic to $B= C(\widehat{G}, A)$, the continuous functions on the dual group $\widehat{G}$ with values in $A$,  and $E$ is indeed  given by the $A$-valued integral $E(f) = \int_{\widehat{G}} f(\gamma) d\gamma$ with respect to the normalized Haar measure on $\widehat{G}$.  The Fourier coefficients of $x \in C^*_r(\Sigma)$ are therefore usually defined by setting $\widehat{x}(g) = E(xu_g^*)$, so the Fourier transform $\widehat{x}$ becomes a function from $G$ to $A$. A useful fact that is not immediately apparent from this definition is that $\widehat{x}$ lies in the space   $$A^\Sigma= \Big\{\xi: G \to A \ | \ \sum_{g \in G}  \alpha_g^{-1}\big(\xi(g)^*\, \xi(g)\big) \ 
\mbox{is norm-convergent in  $A$} \Big\} \,. $$
Note that this statement contains a nontrivial information about the 
decay at infinity of general Fourier coefficients 
(Riemann-Lebesgue Lemma).
In the case of ordinary reduced 
crossed products this was recently observed in \cite[Lemma 5.2]{RoSi}. However, as was already remarked by C. Anantharaman-Delaroche in \cite{AD1}, the 
space $A^\Sigma$ has a natural Hilbert $A$-module structure on which $C_r^*(\Sigma)$ may  be faithfully represented by adjointable operators. The Fourier transform of $x \in C^*_r(\Sigma)$ is then simply  defined by $\widehat{x} = x\, \xi_0 \in A^\Sigma$, where $\xi_0(g) = \delta_{g,e}\, 1$, 
while $E$ is  given on $C_r^*(\Sigma)$ by $E(x) = \widehat{x}(e)$. The (formal) Fourier series of $x$ is now defined as 
$$\sum_{g \in G} \widehat{x}(g) \lambda_\Sigma(g)$$
where the $\lambda_\Sigma(g)$'s denote the canonical unitaries of $C_r^*(\Sigma)$ when it acts on $A^\Sigma$.

Following \cite{BeCo2}, the general problem about norm-convergence of Fourier series in $C_r^*(\Sigma)$ may be considered as the search for ``decay subspaces'' of $A^\Sigma$ that are as large as possible. For example, $\ell^1(G,A)$ is a decay subspace, corresponding to elements of $C_r^*(\Sigma)$ with absolutely convergent Fourier series.
Moreover, inspired by P. Jolissaint's notion of rapid decay (RD) for groups \cite{J}, any weight function $\kappa:G\to [1, \infty)$ such that $G$ is $\kappa$-decaying in the sense of \cite{BeCo2}
leads to a decay subspace $\ell^2_\kappa(G, A)$ of $A^\Sigma$. However, it appears that 
a larger subspace $A^\Sigma_\kappa$ is needed if one wishes to follow the strategy pioneered by U. Haagerup \cite{Haa1} and establish the existence of  certain summation processes, as we did in \cite{BeCo2} for many nonamenable groups.
A problem that appears naturally in 
our framework
is thus to find conditions ensuring that  
$\Sigma$ has the $A^\Sigma_\kappa$-decay property (expressing that $A^\Sigma_\kappa$ is a ``decay subspace''),  
again a kind of generalized version with coefficients of the RD-property. 
When $A$ is commutative and $\alpha$ is trivial, it suffices to assume that $G$ is $\kappa$-decaying (we prove this in the final section), but it should be possible to relax these assumptions. 
R. Ji and L. Schweitzer \cite{JiSc} have a result in this direction for nontrivial actions (and groups of polynomial growth), but it is not obvious to us that their proof can be adapted to give a more general result. 

We next focus on (reduced) multipliers of $\Sigma$. Our ultimate goal in this paper is to use such maps as smoothing kernels for
summation processes for Fourier series, as we did in \cite{BeCo2} when $A=\Complessi$.  Multipliers of C$^*$-dynamical systems are defined in analogy with multipliers on groups, but the terminology ``multiplier" might in fact be somewhat misleading. 
A multiplier $T$ of $\Sigma$ consists of a family $T= \{T_g\}_{g\in G}$ of linear maps from $A$ into itself such that there exists a bounded linear map $M_T$ from $C_r^*(\Sigma)$ into itself satisfying $$\widehat{M_T(x)}(g) =T_g\big(\widehat{x}(g)\big) $$ for all $x \in C_r^*(\Sigma)$ and $g\in G$. Such a multiplier is called a cb-multiplier when the map $M_T$ is completely bounded. Now, given a function $\varphi:G\to A$, one may wonder when it induces a ``left" multiplier and consider the family $T^\varphi= \{T^\varphi_g\}_{g\in G}$ of maps from $A$ to itself given by $T^\varphi_g(a) = \varphi(g)\, a$. We give a set of sufficient conditions ensuring that $T^\varphi$ is a cb-multiplier of $\Sigma$, and use this to show that every cb-multiplier of $G$ induces a cb-multiplier of $\Sigma$. 
We also show that a controlled growth of $\varphi$ w.r.t. a weight $\kappa$, in combination with the  $A^\Sigma_\kappa$-decay property, suffices for $T^\varphi$ to be a multiplier of $\Sigma$. 
In another direction, we give a conceptually satisfactory way of producing cb-multipliers of $\Sigma$, analogous to how matrix coefficients of unitary representations of $G$ induce cb-multipliers on $G$: the maps $T_g:A\to A$ are now of the form
$$T_g(a) = \big\langle x, \, \rho(a)\, v(g)\, y\big\rangle$$
for some equivariant representation $(\rho, v)$ of $\Sigma$ on some Hilbert $A$-module $X$ and $x, y \in X$. 
The collection of these cb-multipliers on $\Sigma$ may therefore be thought of as the analogue of the Fourier-Stieltjes algebra of $G$.
Our proof relies on a new version of Fell's absorption property, that loosely says that any equivariant representation of $\Sigma$ is absorbed when tensoring with some regular covariant representation. (We use here the notion of tensor product introduced in \cite{BeCo3}).  This version complements  the one proven in \cite{BeCo3} about absorption of covariant representations when tensoring with induced regular equivariant representation.

Having studied of multipliers, we turn our attention to summation processes. 
A Fourier summing net for $\Sigma$ is
a net $\{T^i\}$ of multipliers of $\Sigma$ such that,
for each $x \in C^*_r(\Sigma)$, the Fourier series of $M_{T^i}(x)$
$$\sum_{g \in G} T^i_g\big(\hat{x}(g)\big)\lambda_\Sigma(g)\, $$
is norm-convergent (necessarily to $M_{T^i}(x)$) for each $i$, and
$M_{T^i}(x)$ converges in norm to $ x $.
The existence of such Fourier summing nets is then discussed in a 
number of situations. 
In particular, we obtain a generalization of the 
classical Fej\'er summation theorem whenever $\Sigma$ has the weak approximation property of \cite{BeCo3}, and prove some analogs of the Abel-Poisson summation theorem. 
Almost all
the Fourier summing nets we are effectively able to construct have the property that they preserve the invariant ideals of $A$. The existence of such a net affects the ideal structure of $C^*_r(\Sigma)$: $\Sigma$ is then necessarily exact (in the sense of \cite{Si}), and the ideals of $C^*_r(\Sigma)$ that are $E$-invariant are precisely those that are induced from invariant ideals of $A$. Hence, in such a situation, the problem of determining the ideals of $C_r^*(\Sigma)$ reduces  to two different tasks: finding the invariant ideals of $A$ and investigating the possible existence of ideals of $C^*_r(\Sigma)$ that are not $E$-invariant. Especially, one can then deduce that all ideals of $C_r^*(\Sigma)$ are induced from invariant ideals of $A$ if one can show that any ideal of $C_r^*(\Sigma)$ is automatically $E$-invariant, thereby providing an alternative approach to the one obtained in \cite{Si}. We illustrate this with a simple example in the final section.

After having described what we have done in this paper,
we would like to add that we see it
as a first attempt to put 
some facts and ideas into a wider perspective. Many issues remain to be 
addressed, and it might be easier to deal with some of them in specific situations
(depending on various choices of the group, the algebra, the action and the cocycle) 
before attacking the general case. We hope that our work will stimulate
further research on this topic and the reader will find many open questions and 
problems scattered throughout the text. 

The paper is organized as follows.
In Section \ref{Preliminaries} we collect some notions and facts from \cite{BeCo3}.
As the present article is also heavily influenced by the line of thought
presented in \cite{BeCo2}, the reader is kindly advised to have a look at both these articles.
Section \ref{Convergence} is devoted to establishing a first set of
results about convergence of Fourier series.
 The concept of (reduced) multipliers is introduced and discussed in  Section \ref{Multipliers}.
Summation processes for 
Fourier series is the subject of Section \ref{Summation}.
In the last section (Section \ref{Trivial}) we deal with the ``almost trivial'' but still interesting case where $A$ is commutative and $\alpha$ is trivial, and show that in this situation 
the cocycle does not create any trouble for 
the analysis.

\newpage
\section{Preliminaries}\label{Preliminaries}

Throughout the paper, we will use the following conventions. To avoid some technicalities, we will only work in the category of {\it unital}  C$^*$-algebras, and a homomorphism 
between two objects in this category
 will always mean a unit preserving $*$-homomorphism.  Isomorphisms and automorphisms are consequently 
 also assumed to be $*$-preserving. The group of unitary elements in a C$^*$-algebra $A$ will be denoted by $\U(A)$, the center of $A$ by $Z(A)$, while the group of automorphisms of $A$ will be denoted by ${\rm Aut}(A)$. The identity map on $A$ will be denoted by ${\rm id}$ (or ${\rm id}_A$). By an ideal of $A$, we will always mean a two-sided closed ideal, unless otherwise specified. If $B$ is another C$^*$-algebra, $A\otimes B$ will denote their minimal tensor product.

\medskip By a Hilbert C$^*$-module, we will always mean a  {\it right} Hilbert C$^*$-module and follow the notation introduced in \cite{La1}. Especially, all inner products will be assumed to be linear in the second variable, $\L(X, Y)$  will denote the space of all adjointable operators between two Hilbert C$^*$-modules $X$ and $Y$ over a C$^*$-algebra $B$, and $\L(X) = \L(X,X)$. A representation of C$^*$-algebra $A$ on a Hilbert $B$-module $Y$ is then a homomorphism from $A$ into the C$^*$-algebra $\L(Y)$. If $Z$ is another Hilbert C$^*$-module (over $C$), we will let $\pi \otimes \iota : A  \to \L(Y \otimes Z)$ denote the amplified representation of $A$ on $Y \otimes Z$ given by $(\pi \otimes \iota)(a) = \pi(a) \otimes I_Z$, where the Hilbert $B\otimes C$-module $Y \otimes Z$ is the external tensor product of $Y$ and $Z$ and $I_Z$ denotes the identity operator on $Z$.  Note that if $Z$ is a Hilbert space, i.e. a Hilbert $\Complessi$-module, then we may and will regard $Y\otimes Z$ as a Hilbert $B$-module. 

\bigskip We will work with series in a C$^*$-algebra $A$ of the form $\sum_{i \in I} a_i$ where $I$ is a possibly uncountable set and $a_i \in A$ for each $i \in I$. Norm-convergence of such a series will always mean unconditional convergence (sometimes called summability). Since $A$ is Banach space, this happens if and only the usual Cauchy criterion is satisfied \cite{Di3}. An immediate consequence is the following fact, used without notice on several occasions in the sequel: if $\{a_i\}_{i \in I}, \,   \{b_i\}_{i \in I}$ are families of  elements in $A^{+}$ (the cone of positive elements in $A$) such that $a_i \leq b_i$ for each $i \in I$ and $\sum_{i \in I} b_i$ is norm-convergent to $b$ (lying necessarily in $A^{+}$), then $\sum_{i \in I} a_i$ is norm-convergent to some $a$ in $A^{+}$ satisfying $a \leq b$ (hence also $\|a\| \leq \|b\|)$.

\bigskip 
The quadruple $\Sigma = (A, G, \alpha,\sigma)$ will always denote a {\it twisted 
$($unital, discrete$)$ 
C$^*$-dynamical system}.  This means that
$A$ is a  C$^*$-algebra with unit $1$, 
$G$ is a discrete group with identity $e$
and $(\alpha,\sigma)$ is a {\it twisted
action} of $G$ on $A$ (sometimes called a cocycle $G$-action on $A$), that is,
$\alpha$ is a map from $G$ into ${\rm Aut}(A)$ 
and  $\sigma$ is a map from $G \times G$ into $\, \U(A)$,
satisfying
\begin{align*}
\alpha_g \, \alpha_h & = {\rm Ad}(\sigma(g, h)) \,  \alpha_{gh} \\
\sigma(g,h) \sigma(gh,k) & = \alpha_g(\sigma(h,k)) \sigma(g,hk) \\
\sigma(g,e) & = \sigma(e,g) = 1 \ , 
\end{align*}
for all $g,h,k \in G$. Of course,  ${\rm Ad}(v)$ denotes here the (inner) automorphism  of $A$ implemented by some unitary $v$ in $\U(A)$.

\medskip    
If $\sigma$ is  trivial, that is, $\sigma(g,h)=1$ for all $g,h \in G$, then $\Sigma$ is an ordinary C$^*$-dynamical system (see e.g.\ \cite{Wi, BrOz, Ec}), and one just writes $\Sigma=(A, G, \alpha)$.  If $\sigma$ is {\it central}, that is, takes values in  $\U(Z(A))$, then $\alpha$ is an ordinary  action of $G$ on $A$, and  this is the case studied in \cite{ZM}. If $A = {\mathbb C}$, then $\alpha_g={\rm{id}}$ for all $g \in G$ and $\sigma$ is a 2-cocycle on $G$ with values in the unit circle $\mathbb{T}$,
 (see e.g.\ \cite{BeCo2} and references therein).    

\medskip
To each twisted C$^*$-dynamical system  $\Sigma = (A, G, \alpha,\sigma)$ 
one may associate its {\it full twisted crossed product} C$^*$-algebra
$C^*(\Sigma)$ 
and its {\it reduced} version 
$C^*_r(\Sigma)$ (see \cite{PaRa, PaRa1}). In this paper we will be mostly interested in the reduced algebra. For the ease of the reader, we will recall some definitions and facts from \cite{BeCo3} needed in the sequel.

\medskip A {\it covariant homomorphism} of $\Sigma$ is a pair $(\pi,u)$,
where $\pi$ is a 
homomorphism of $A$ into a  
C$^*$-algebra $C$ and  $u$ is a map of $G$ into $\U(C)$, which satisfy
$$u(g)\, u(h) = \pi(\sigma(g,h))\,  u(gh)$$ and
the covariance relation 
\begin{equation}
\pi(\alpha_g(a)) = u(g) \, \pi(a) \, u(g)^* 
\end{equation}
 for all $g,h \in G$, $a \in A$. 
 Every such a pair induces a unique canonical homomorphim $\pi \times u$ from $C^*(\Sigma)$ onto the C$^*$-subalgebra of $C$ generated by $\pi(A)$ and $u(G)$.
 If $C=\L(X)$ for some Hilbert C$^*$-module $X$, then 
  $(\pi, u)$ is called a {\it covariant representation} of $\Sigma$ on $X$.

\medskip  Let $Y$ be a Hilbert $B$-module 
and assume $\pi$ is a 
representation of $A$ on $Y$.
We can then form  the  Hilbert $B$-module $Y^G$ ($\simeq Y\otimes \ell^2(G)$) given by
\begin{equation}
Y^G 
= \Big\{\xi: G \to Y \ | \ \sum_{g \in G} \big\langle\xi(g),\xi(g)\big\rangle \ 
\mbox{is norm-convergent in  $B$} \Big\} \, 
\end{equation}
endowed with the $B$-valued scalar product 
$$\big\langle\xi,\eta \big\rangle = \sum_{g\in G} \big\langle \xi(g),\eta(g)\big\rangle$$ and the 
natural module right action of $B$ given by
$$(\xi\cdot b )(g) = \xi(g)\, b \,, \quad g \in G\,.$$

The  {\it regular covariant  representation} 
$(\tilde{\pi},  \tilde{\lambda}_\pi)$ of $\Sigma$ on $Y^G$ associated to $\pi$ is then
defined by
\begin{align}
(\tilde{\pi}(a)\xi)(h) & = \pi\big(\alpha_h^{-1}(a)\big)\xi(h) \ , 
\hspace{10ex}  a \in A, \, \xi \in Y^G, \, h \in G, \ \\
(  \tilde{\lambda}_\pi(g)\xi)(h) & = 
\pi\big(\alpha_h^{-1}(\sigma(g,g^{-1}h))\big)\xi(g^{-1}h) \ , 
\quad g,h \in G, \, \xi \in Y^G \, .
\end{align}

\medskip Considering  $A$ as a Hilbert $A$-module in the standard way and letting $\ell: A \to \L(A)$  be given  by $\ell(a)(a')= a  a' , \, a, a' \in A$, we get {\it the regular covariant  representation} 
$(\tilde{\ell}, \tilde{\lambda}_\ell)$
  associated to $\ell$, that
   acts on the Hilbert $A$-module  
   \begin{equation}
A^G 
= \Big\{\xi: G \to A \ | \ \sum_{g \in G} \, \xi(g)^*\,\xi(g) \ 
\mbox{is norm-convergent in  $A$}\Big\}
\end{equation}
 in the following way:
\begin{align}
(\tilde{\ell}(a)\xi)(h) & = \alpha_h^{-1}(a) \,\xi(h) \ , 
\hspace{10ex}  a \in A, \, \xi \in A^G, \, h \in G, \ \\
(  \tilde{\lambda}_\ell(g)\xi)(h) & = 
\alpha_h^{-1}(\sigma(g,g^{-1}h))\,\xi(g^{-1}h) \ , 
\quad g,h \in G, \, \xi \in A^G \, .
\end{align}

The {\it reduced twisted crossed product} $C^*_r(\Sigma)$ is defined as the C$^*$-subalgebra of $\L(A^G)$ generated by $\tilde{\ell}(A)$ and $\tilde{\lambda}_\ell(G)$. 

\smallskip Setting $\Lambda = \tilde{\ell}\times \tilde{\lambda}_\ell$, we have
  $C^*_r(\Sigma) = \Lambda(C^*(\Sigma))$.
Moreover, $C^*_r(\Sigma) \simeq (\tilde{\pi}\times   \tilde{\lambda}_\pi)(C^*(\Sigma))$ whenever $\pi: A \to \L(Y)$ is a faithful representation of $A$ on any Hilbert C$^*$-module $Y$ (e.g. a Hilbert space).

\medskip 

It turns out to be useful to also consider the Hilbert $A$-module 
  $$A^\Sigma= \Big\{\xi: G \to A \ | \ \sum_{g \in G}  \alpha_g^{-1}\big(\xi(g)^*\, \xi(g)\big) \ 
\mbox{is norm-convergent in  $A$} \Big\} \, ,$$
where 
the right action of $A$ on $A^\Sigma$ and the $A$-valued scalar product are defined by
$$(\xi \times a )(g) = \xi(g)\,\alpha_g(a)\, ,$$
$$\langle\xi,\eta \rangle_\alpha = \sum_{g\in G} \alpha_g^{-1}\big(\xi(g)^*\eta(g)\big)\, ,$$
the associated norm on $\A^\Sigma$ being given by \,
 $\|\xi\|_\alpha = \big\| \,\sum_{g\in G} \alpha_g^{-1}\big(\xi(g)^*\xi(g)\big) \,\big\|^{1/2}\,.$

\smallskip As $A^G$ and $A^\Sigma$ are unitarily equivalent via 
the unitary operator $J:A^G \to A^\Sigma$ given by
$$(J\eta)(g)=\alpha_g(\eta(g))\, , \quad \eta \in A^G, \, g \in G\,,$$
we get a covariant representation $(\ell_\Sigma, \lambda_\Sigma)$ of $\Sigma$ on $A^\Sigma$ given by $$\ell_\Sigma(a) = J \,\tilde{\ell}(a) \, J^* \, , \quad  \lambda_\Sigma(g) =  J \,\tilde{\lambda}_\ell(g) \, J^* \, ,$$
that is, 
\begin{equation*}\label{LS}\big(\ell_\Sigma(a)\xi\big)(h)  = a\, \xi(h)\, , 
\end{equation*}
\begin{equation*}\label{laS}
\big(\lambda_\Sigma(g)\xi\big)(h) = 
\alpha_g(\xi(g^{-1}h))\, \sigma(g, g^{-1}h)\, ,
\end{equation*}
where $a \in A, \, \xi \in A^\Sigma, \, g,\, h \in G$.

\medskip As $\Lambda_\Sigma= \ell_\Sigma \times \lambda_\Sigma$  is 
unitarily equivalent to $\Lambda$, we may identify $C_r^*(\Sigma)$ with $\Lambda_\Sigma(C^*(\Sigma))$. Further, we may also identify $A$ with $\ell_\Sigma(A)$, so $A$ acts on $A^\Sigma$ via  $$(a\, \xi)(h)  = a\, \xi(h)\, , \, \, a\in A\, , \, \xi \in A^\Sigma\, , \, h \in H\,.$$
Letting $C_c(\Sigma)$
denote the set of functions from 
$G$ into $A$ with finite support, and identifying it with its canonical copy inside $C^*(\Sigma)$, we get
$$ \Lambda_\Sigma(f) = \sum_{g\, \in\,  \text{supp}(f)} f(g) \, \lambda_{\Sigma} (g)\,,  \quad f \in C_c(\Sigma)\, .$$
Especially, letting  $a \odot  \delta_g$  denote the function in  $C_{c}(\Sigma)$ which is 
 0 everywhere except at the point $g \in G$ where it takes the value $a \in A$, we have
 $$  \Lambda_\Sigma(a \odot  \delta_g) = a \, \lambda_{\Sigma} (g)\,.$$

\medskip The {\it Fourier transform} is  the (injective, linear) map $\, x \to \widehat{x}$ from $C^*_r(\Sigma)$ into $A^\Sigma$ given by
 $$\widehat{x} = x \, \xi_0\,$$ 
where $\xi_0 = 1\odot \delta_e \in A^\Sigma$. 

\medskip When  $f \in C_c(\Sigma)$ and  $x \in C^*_r(\Sigma)$, we have
 \begin{equation}\label{norm-ineq}
 \widehat{\Lambda_\Sigma(f)} = f \, \, , \quad  \quad \|\widehat{x}\|_\infty \leq \|\widehat{x}\|_\alpha \leq \|x\|\, ,
 \end{equation}
where $\, \|\widehat{x}\|_\infty = \sup_{g\in G} \|\widehat{x}(g)\|$.

\medskip The {\it canonical  conditional expectation} $E$ from $C^*_r(\Sigma)$ onto $A$ 
is given by  
$E(x) = \widehat{x}(e) $. 
It satisfies that $E(\Lambda_\Sigma(f)) = f(e)\, , \, f\in C_c(\Sigma)$. 
Moreover, we have 
$$E(x^*x) = \|\widehat{x}\|_\alpha^{\,2}\, \,, \quad \quad E( x\, \lambda_\Sigma(g)^*) = \widehat{x}(g)\,\,  ,\quad \quad E\big(\lambda_\Sigma(g)\, x \,\lambda_\Sigma(g)^*\big) = \alpha_g(E(x))$$  
 for all $x \in C^*_r(\Sigma), \, g\in G$.

 \bigskip Another  concept, slightly adapted from \cite{EKQR-0}, that will be of importance to us is the following: An {\it equivariant representation} 
of $\Sigma$ on a Hilbert $A$-module $X$ 
is  a pair $(\rho, v)$ where 
 $\rho : A \to \L(X)$ is a representation of $A$ on $X$ and  $v$ is a map from $G$ into the group $\mathcal{I}(X)$ consisting of all $\mathbb{C}$-linear, invertible, bounded maps from $X$ into itself, which  satisfy:
\begin{itemize}
\item[(i)]  \quad $\rho(\alpha_g(a))  = v(g) \, \rho(a) \, v(g)^{-1}\, , \quad  \quad g\in G\,, \,  a \in A$
\item[(ii)]  \quad$v(g)\, v(h)  = {\rm ad}_\rho(\sigma(g,h)) \, v(gh) \, , \quad   \quad g, h \in G$
\item[(iii)] \quad $\alpha_g\big(\langle x \, ,\, x' \rangle\big)   = \langle v(g) x\, ,\, v(g) x' \rangle\, , \quad  \quad g\in G\, , \, x,\,  x' \in X\,$ 
\item[(iv)]  \quad$v(g)(x \cdot a)  = (v(g) x)\cdot \alpha_g(a)\, , 
\quad  \quad \, g \in G,\, x\in X,\, a \in A$. 
\end{itemize}
In (ii) above, $ {\rm ad}_\rho(\sigma(g,h)) \in \mathcal{I}(X) $ is defined by
$${\rm ad}_\rho(\sigma(g,h)) \,x = \big(\rho(\sigma(g,h))\, x \big)\cdot \sigma(g,h)^* \, , \quad g, h \in G,\, x \in X. $$
The central part of $X$ is defined by 
$$Z_X = \{z \in X \mid \rho(a)z = z\cdot a \, \, \text{for all} \, \, a \in A\}\, .$$
An important feature is that whenever  $(\pi,u)$ is  a covariant representation of $\Sigma$ on some Hilbert $B$-module $Y$,  we can form the product covariant representation  $(\rho\dot\otimes\pi\, , \, v\dot\otimes u)$ of $\Sigma$ on the Hilbert $B$-module $X\otimes_\pi Y$, see \cite[Section 4]{BeCo3}.

\medskip The  {\it trivial equivariant representation} of $\Sigma$ is  the pair $(\ell, \alpha)$ acting on the $A$-module $A$ 
(with its canonical structure). 
The {\it regular equivariant representation} of $\Sigma$ is the pair $(\check{\ell}, \check{\alpha})$  on $A^G$ defined by
 $$(\check{\ell}(a)\, \xi)(h) = a\, \xi(h)\, \quad 
(\check{\alpha}(g)\,\xi)(h) = \alpha_g(\xi(g^{-1}h))$$
where $a \in A, \, \xi \in A^G, \, g, h \in G$. 
More generally, if  $(\rho,v)$ is an equivariant 
representation of $\Sigma$ on a Hilbert $A$-module $X$,
it induces an equivariant 
representation $(\check{\rho},\check{v})$ 
of $\Sigma$ on $X^G$ given  by
$$(\check{\rho}(a)\xi)(h)  = \rho(a)\xi(h), \quad
(\check{v}(g)\xi)(h)  = v(g)\xi(g^{-1}h) \ ,
$$
for all $a \in A, \, \xi \in X^G,\, g,h \in G$.

\medskip 
 We recall from \cite{BeCo3} that $\Sigma$ is said to have the  {\it weak approximation property} if there exist an equivariant representation $(\rho, v)$ of $\Sigma$ on some $A$-module $X$ and 
nets $\{\xi_i\}, \ \{\eta_i\} $ in $ X^G$, (that  both may be chosen with finite support) satisfying
\begin{itemize}
\item[a)] there exists some $ M > 0$ such that $\|\xi_i\| \cdot \|\eta_i\| \leq M $ \ for all $i $;
\item[b)] for all $g \in G$ and $a \in A$ we have $\lim_i \| \big\langle \xi_i\,,\,\check\rho(a)\check{v}(g)\eta_i \big\rangle  - a\|= 0$, i.e.,
$$\, \lim_i   \sum_{h \in G} \Big\langle \xi_i(h)\, , \,  \rho(a) \,v(g)\eta_i(g^{-1}h)\Big\rangle \, =\,  a\, .$$
\end{itemize}
As shown in \cite[Theorem 5.11]{BeCo3},  the weak approximation property is enough to ensure regularity of $\Sigma$, that is, $\Lambda: C^*(\Sigma) \to C_r^*(\Sigma)$ is then an isomorphism. 

\medskip If $(\rho, v)$ can be chosen to be equal to $(\ell, \alpha)$ in the above definition, one recovers the {\it  approximation property} introduced by Exel \cite{Ex} (see also \cite{ExNg}).  

\smallskip If $\{\xi_i\}$ or $ \{\eta_i\} $ (resp. $\{\xi_i\}$ and $ \{\eta_i\} $) can be chosen to lie in the central part of $X^G$, we will say that $\Sigma$ has the  {\it half-central} (resp. {\it central}) {\it weak approximation property}. 
See Remarks 5.9 and 5.10 in \cite{BeCo3} for a discussion of other related notions.

\section{Convergence of Fourier series}\label{Convergence}
Given $x \in C^*_r(\Sigma)\subset \L(A^\Sigma)$, 
its (formal) {\it Fourier series} is defined by
$$\sum_{g \in G}\, \widehat{x}(g)\, \lambda_\Sigma(g)\,.$$ 
It is well known that this series 
will not necessarily be convergent w.r.t.\ the operator norm $\|\cdot\|$ on $\L(A^\Sigma)$ (even in the classical case where $A,\,  \alpha$ and $\sigma$ are all trivial and $G$ is abelian).

\medskip  However, if we consider the norm on $C^*_r(\Sigma)$ given by  
$\, \|x\|_\alpha =  \|\widehat{x}\|_\alpha \, ,$
then the Fourier series of $x \in C^*_r(\Sigma)$ converges to $x$ w.r.t. 
$\|\cdot\|_\alpha$\,.

\medskip Indeed, for $F \subset G$, $F$ finite,  set  $x_F=  \sum_{g \in F} 
\widehat{x}(g)\,\Lambda_\Sigma(g)$. Letting $\chi_F$ denote the characteristic function of $F$ in $G$, we have
$$\widehat{x_F}(g) = 
\begin{cases}\widehat{x}(g), & g \in F \\ \, \, 0\quad ,  & g\not\in F \end{cases}\, \,  = \, \big(\widehat{x}\,  \chi_F\big)(g)$$
for all $g \in G$.
It  follows 
 that $\|x_F -x\|_\alpha = \|\widehat{x}\,  \chi_F - \widehat{x}\|_\alpha \to 0 $ as $F\uparrow G$. 
 
  \medskip
For later use we also record a related fact.

\begin{proposition}\label{Four}
Let $\xi: G \to A$ and assume that $\sum_{g\in G} 
\xi(g)\lambda_\Sigma(g)$ converges to some $x \in 
C^*_r(\Sigma)$ w.r.t. $\|\cdot\|_\alpha$. 
Then $\xi \in A^\Sigma$ and $\xi = \widehat{x}$.
\end{proposition}

\begin{proof} 
Let  $F$ be a finite subset  of $G$ and set
 $$y_F = \sum_{g \in F} \xi(g)\lambda_\Sigma(g)\,.$$ 
Then $\widehat{y_F} = \xi_F$ where $\xi_F=\xi \, \chi_F$.
 As $y_F \to x$ w.r.t. $\|\cdot\|_\alpha$ by assumption,
we have $\xi_F = \widehat{y_F} \to \widehat{x}$ w.r.t. $\|\cdot\|_\alpha$.
Hence, for every $g \in G$,  we get
$$\|(\widehat{x}-\xi_F)(g)\|^2 = \| \alpha_g^{-1}\big((\widehat{x}-\xi_F)(g)^* (\widehat{x} -\xi_F)(g))\big)\| 
\leq \|\widehat{x}-\xi_F\|^2_\alpha \, \to 0\, $$
as $F \uparrow G$. 
That is, $\xi_F(g) \to \widehat{x}(g)$ in norm for every $g\in G$, which gives
$$\xi(g) = \lim_{F \uparrow G} \xi_F(g) = \widehat{x}(g), \quad g \in G \,. $$
So $\xi= \widehat{x} \in A^\Sigma$, as asserted.
\end{proof}
We set 
$$CF(\Sigma)= \Big\{ x\in C_r^*(\Sigma)\, \, \big| \, \, \sum_{g \in G}\, \widehat{x}(g)\, \lambda_\Sigma(g) \, \textup{is convergent 
w.r.t.} \, \|\cdot\| \Big\}\,.$$

\medskip Note that if $x \in CF(\Sigma)$, then, as $\| \cdot\|_\alpha \leq \|\cdot\| $, it readily follows from 
what we just have seen
that the Fourier series of $x$ necessarily converges to $x$ w.r.t.\ $\|\cdot\|$. 
In order to describe some subspaces of  $CF(\Sigma)$, we adapt some definitions from \cite{BeCo2}.

\medskip Let $\L$ be a subspace of $A^\Sigma $ containing 
$C_c(\Sigma)$
and let $\|\cdot \|'$ be a norm on $\L$.
If $\xi \in \L$, then we will say that $\xi \to 0$ {\it at infinity} (w.r.t. 
$\|\cdot\|'$) when, for every $\epsilon >0$, 
there exists a finite subset $F_0$ of $G$ such that 
$\|\xi_F\|' < \epsilon$ for all finite subsets $F$ disjoint from $F_0$. 
We will also say that $\Sigma$ has the {\it $\L$-decay property} (w.r.t. $\|\cdot\|'$)
if $\xi \to 0$ at infinity for every $\xi \in \L$
and 
there exists some $C > 0$ such that
\begin{equation} \label{C}
\|\Lambda_\Sigma(f)\| \leq C \, \|f\|', \quad f \in C_c(\Sigma) \,. 
\end{equation}

\begin{proposition} \label{L-decay}
Let $\Sigma$ have the $\L$-decay property $($w.r.t. \ $\|\cdot\|'$$)$ and $\xi \in \L$.

\medskip \noindent Then $\sum_{g \in G} \xi(g)\, \lambda_\Sigma(g)$ converges in 
operator norm to some $x \in C^*_r(\Sigma)$ satisfying $\widehat{x} = \xi$.

\medskip \noindent Denoting this $x$ by $\tilde\Lambda_\Sigma(\xi)$, and
letting $\tilde\Lambda_\Sigma: \L \to C^*_r(\Sigma)$ be the associated 
map,
we have
\begin{equation}
\tilde\Lambda_\Sigma(\L) = \big\{x \in C^*_r(\Sigma) \ | \ \widehat{x} \in \L\big\} 
\subseteq CF(\Sigma) \ . 
\end{equation}
\end{proposition}

\begin{proof}
This proposition is the direct analogue of \cite[Lemma 3.4]{BeCo2} and \cite[Theorem 3.5]{BeCo2}, and their proofs
 adapt in a verbatim way (using Proposition \ref{Four} instead of \cite[Proposition 2.10]{BeCo2}). 

\end{proof}

For example, consider 
$$\ell^1(G,A) = \{\xi : G \to A \ | \ \sum_{g \in G} \|\xi(g)\| < 
\infty\}\, $$ and let $\|\cdot\|_1$ denote the associated norm. Then, clearly, $\ell^1(G,A)$ is a subspace of $A^\Sigma$. Moreover,   $\Sigma$ has the $\ell^1(G,A)$-decay property (w.r.t. $\|\cdot\|_1$): $\xi \to 0$ at infinity for every $\xi \in  \ell^1(G,A)$ (since this property holds in $\ell^1(G)$) and
$$\| \Lambda_\Sigma(f)\| \leq \sum_{g \,\in \,\text{supp}(f)} \|f(g)\, \lambda_\Sigma(g)\| =\sum_{g \,\in \,\text{supp}(f)} \|f(g)\| = \| f\|_1 $$
holds for every $f \in C_c(\Sigma)$.

\medskip The space $\ell^2(G,A) = \{\xi : G \to A \ | \ \sum_{g \in G} \|\xi(g)\|^2 < 
\infty\}\,$, 
equipped with its natural norm $\|\cdot\|_2$, is also a subspace of $A^\Sigma$, but it can not be expected that  $\Sigma$ will have the $\ell^2(G,A)$-decay property (as this is not true when $A=\Complessi$, unless if $G$ is finite).
 
 \smallskip We may instead consider weighted $\ell^2$-spaces. Dealing only with the scalar-valued case, we 
pick some function $\kappa: G \to [1,+\infty)$
and equip $$\ell^2_\kappa(G,A) = \big\{\xi: G \to A \ | \ \sum_g \|\xi(g)\|^2 \kappa(g)^2 < + \infty\big\}$$
with its natural norm $\|\xi\|_{2,\kappa} = \|\xi\, \kappa\|_2\,$.

\medskip For example, assume that  $\kappa^{-1} \in \ell^2(G)$.  Then $\Sigma$  
has the $\ell^2_\kappa(G,A)$-decay 
property (w.r.t. $\|\cdot\|_{2,\kappa}$).
Indeed, it follows  from the Cauchy-Schwarz inequality that
$$\|\Lambda_\Sigma(f)\| \leq \|f\|_1 \leq  \|\kappa^{-1} \|_2 \, \|f\|_{2,\kappa}\, \quad $$
for every $f \in C_c(\Sigma)$, and the assertion easily follows.

\bigskip

In \cite{BeCo2}, we introduced the notion of {\it $\kappa$-decay} for the group $G$. It just expresses that
the system $(\Complessi, G, {\rm id}, 1)$ has the $\ell^2_\kappa(G,\Complessi)$-decay property (w.r.t. $\|\cdot\|_{2,\kappa}$). Note that the inequality (\ref{C}) then just amounts to $$\| f*g\|_2 \leq C \, \|f\|_{2, \kappa}\, \|g\|_2$$
for all $f\in C_c(G)\, $ and $ \, g \in \ell^2(G)$ (where $f*g$ denotes the usual convolution product), and the least possible $C>0$ is called the {\it decay constant}. One should note that any countable group is $\kappa$-decaying for suitable choices of $\kappa$ with relatively slow growth, see for example
\cite[Lemma 3.14]{BeCo2}.

\begin{proposition} \label{k-decay}
Suppose that $G$ is $\kappa$-decaying. 
Then $\Sigma$ 
has the $\ell^2_\kappa(G,A)$-decay property $($w.r.t. $\|\cdot\|_{2,\kappa}$$)$. 
Hence, if $x \in C^*_r(\Sigma)$ and $\widehat{x} \in \ell^2_\kappa(G,A)$, then $x \in CF(\Sigma)$.
\end{proposition}

\begin{proof} If $\xi \in \ell^2_\kappa(G,A)$, then $\xi \to 0$ at infinity w.r.t. $\|\cdot\|_{2, \kappa}$ (since $g \to\|(\xi\kappa)(g)\|$ is a function in $\ell^2(G)$ and therefore goes to 0 at infinity). 

\smallskip \noindent Next, we pick a faithful representation $\pi$  of $A$ on some Hilbert space $\H$ with associated norm $\|\cdot\|_\H$, and  form the regular representation
 $\tilde{\pi}\times\tilde{\lambda}_\pi$ of $C^*(\Sigma)$ on $\tilde\H=\ell^2(G, \H)$ with associated norm given by $\|\eta\|_{\H}=(\sum_{g \in G} \|\eta(g)\|^2_\H)^{1/2}$. 
 
\medskip As $ \| \Lambda_\Sigma(f)\| = \| \Lambda(f)\|= \|(\tilde{\pi}\times\tilde{\lambda}_\pi) f\|$ for all $f \in C_c(\Sigma)$, it suffices to show that
there exists some $C>0$ such that for all $f \in C_c(\Sigma)$ we have $$\|(\tilde{\pi}\times\tilde{\lambda}_\pi) f\| \leq C\, \|f\|_{2,\kappa}\,.$$
So let $f \in C_c(\Sigma)$. Set $F ={\rm supp}(f)$ and define  $\tilde{f}(g) = \|f(g)\|\, , \, g\in G$, 
so $\tilde{f} \in C_c(G)$.
Then
$$\|f\|_{2,\kappa}^2 = \sum_{g\in F} \|f(g)\kappa(g)\|^2 
= \sum_{g\in F} |\tilde{f}(g)|^2 \kappa(g)^2 = \|\tilde{f}\|_{2,\kappa}\,.$$
Let $\xi \in \tilde\H=\ell^2(G,\H)$ and set $\tilde{\xi}(g) = \|\xi(g)\|_\H\, , \, g\in G$,
so $\tilde\xi \in \ell^2(G)$ and $\|\tilde\xi\|_2 = \|\xi\|_{\tilde\H}$.

\medskip 

\noindent Let now
$h \in G$.
Then we have
\begin{align*}
\|[((\tilde{\pi}\times\tilde{\lambda}_\pi)f)\xi](h)\|_\H & =
\Big\|\Big(\sum_{g\in F}\,  \tilde{\pi}(f(g)) \, \tilde{\lambda}_\pi(g)\xi\Big)(h)\Big\|_\H 
\\
& = \Big\|\sum_{g\in F}\,  \pi\big(\alpha^{-1}_h(f(g))\big) 
\pi\big(\alpha^{-1}_h(\sigma(g,g^{-1}h))\big) 
\xi(g^{-1}h) \Big\|_\H \\
& \leq \sum_{g\in F}\,  \big\| \pi\big(\alpha^{-1}_h(f(g))\big) \big\| \; 
\big\| \pi\big(\alpha^{-1}_h(\sigma(g,g^{-1}h))\big) 
\xi(g^{-1}h) \big\|_\H \\
& = \sum_{g\in F} \, \|f(g)\| \; \| \xi(g^{-1}h)\|_\H  = \sum_{g\in F} \, \tilde{f}(g) \, \tilde\xi(g^{-1}h) \\
& = (\tilde{f} * \tilde\xi)(h) \ . 
\end{align*}

\noindent This gives
\begin{align*}
\Big\| ((\tilde{\pi}\times\tilde{\lambda}_\pi)f) \xi \Big\|_{\tilde\H}^2 & 
\leq \sum_{h \in G} |(\tilde{f} \ast \tilde\xi)(h)|^2  = \|\tilde{f} \ast \tilde\xi\|_2^2 \\
& \leq C^{\,2} \, \|\tilde{f}\|_{2,\kappa}^2 \, \|\tilde{\xi}\|_2^2  = C^{\,2} \,\| f \|_{2,\kappa}^2 \, \|\xi\|_{\tilde\H}^2 \ ,
\end{align*}
where  $C$ denotes
the decay constant of $G$ w.r.t. $\kappa$.
Hence, 
\begin{equation*}
\|(\tilde{\pi}\times\tilde{\lambda}_\pi)f\| \leq C \|f\|_{2,\kappa}
\end{equation*}
as desired. The last assertion then follows from Proposition \ref{L-decay}. 

\end{proof}

We may now obtain a generalized version of \cite[Theorem 3.15]{BeCo2}. It relies on the concept of polynomial (resp.\ subexponential) H-growth that we introduced in \cite{BeCo2}. For amenable groups these concepts of growth coincide with the classical ones. For other examples, 
see \cite[Section 3]{BeCo2}. 

\begin{corollary} \label{M1} Let $L: G\to [0, \infty)$ be a proper function. 
   
   \smallskip If $G$ has polynomial H-growth $($w.r.t.\ $L$$)$, then there exists some $s >0 $ such that 
   the Fourier series of 
   $x \in C_{r}^*(\Sigma)$ 
   converges  to $x$
   in operator norm  whenever $$\sum_{g\in G} \| \widehat{x}(g)\|^2 \, (1 + L(g))^{s} < \infty \,.$$
   
   If $G$ has subexponential H-growth (w.r.t.\ $L$), then the Fourier series of $x \in C_{r}^*(\Sigma)$ converges  to $x$
   in operator norm  whenever there exists some $t > 0$ such that  $$\sum_{g\in G} \|\widehat{x}(g)\|^2 \,\exp(t L(g)) < \infty \,.$$
   \end{corollary} 

\begin{proof} Assume that $G$ has polynomial H-growth (w.r.t.\ $L$). By \cite[Theorem 3.13, part 1)]{BeCo2}, we know that $G$ is $(1+L)^{s}$-decaying for some $s >0$. The first assertion is then a direct consequence of Proposition \ref{k-decay}.
If $G$ has subexponential H-growth (w.r.t.\ $L$), the proof is similar, except that we now use \cite[Theorem 3.13, part 2)]{BeCo2}.

\end{proof}

In some aspects (see for example Proposition \ref{CK} and its following remark), 
the spaces $\ell^2_\kappa(G,A)$, as subspaces of $A^\Sigma$, seem to be too small when $A$ is non-trivial, and one should instead consider the subspaces
\begin{align*}
A_{\kappa}^\Sigma &= \Big\{\xi: G \to A \ | \ 
\sum_{g \in G} \alpha_g^{-1}(\xi(g^*)\xi(g)) \kappa(g)^2
\ \mbox{is norm-convergent in $A$}\Big\}\, \\
&= \big\{\xi: G \to A \ | \ \xi\, \kappa \in A^\Sigma\, \big\}\,
\end{align*}
equipped with the norm
$\|\xi\|_{\alpha,\kappa} = \| \xi \kappa\|_\alpha$. In the case where $A$ is commutative, $\sigma$ is scalar-valued and $\kappa$ is of the form $\kappa=(1+L)^m$ for some proper length function $L$ on $G$ and $m\in \Naturali$, such subspaces have previously been considered in \cite{JiSc, CW}. 
Note that we have
$$
\begin{array}{ccccc}
\ell^1(G,A) & \subset & \ell^2(G,A) & \subset & A^\Sigma  \\
& & \cup & & \cup \\
& & \ell^2_\kappa(G,A) & \subset & A_{\kappa}^\Sigma 
\end{array}
$$
Moreover, if $\Sigma$ has the $A_{\kappa}^\Sigma$-decay property,
then $\Sigma$ also has the $\ell^2_\kappa(G,A)$-decay property (as  
$ \|f\|_{\alpha,\kappa} \leq  \|f\|_{2,\kappa}, 
\, f \in C_c(\Sigma)$).

We will see in Corollary \ref{comdecay} that $\Sigma$ has the  $A_{\kappa}^\Sigma$-decay property whenever $A$ is commutative, $\alpha$ is trivial and $G$ is $\kappa$-decaying.
We expect that the $A_{\kappa}^\Sigma$-decay property  will also hold in some
cases where the action is not trivial, but leave this open for future investigations.
We only mention that it might be useful to consider the notion of {\it $\Sigma$-content}, defined for a  finite nonempty  subset $E$ of $G$ by
$$C_\Sigma(E) = \sup\big\{ \|\Lambda_\Sigma(f)\| \ | \ 
f \in C_c(\Sigma), \, {\rm supp}(f) \subseteq E, \  \|f\|_\alpha = 1 \big\} \ . $$
To see that $C_\Sigma(E)$ is finite, consider $f \in C_c(\Sigma)$  satisfying ${\rm supp}(f) \subseteq E$ and $\|f\|_\alpha = 1$.
As $\|f(g)\| \leq \|f\|_\alpha=1$ for all $g\in G$, we have
$$\|\Lambda_\Sigma(f)\| 
\leq \sum_{g\in G} \, \|f(g)\| \leq |E| \, \|f\|_\alpha\, = |E| .$$
Hence, $C_\Sigma(E) \leq |E| < \infty$. One can also check that if $F$ is another finite nonempty subset of $G$, then we have $C_\Sigma(E) \leq C_\Sigma(F)$ whenever $E\subseteq F$, and  $C_\Sigma(E\cup F) \leq C_\Sigma(E) + C_\Sigma(F)$ whenever
$E$ and $F$ are disjoint. 

When $A=\Complessi$, then it is not difficult to see that $C_\Sigma(E) \leq c(E) \leq |E|^{1/2}$, where  $c(E)$ denotes the Haagerup content of $E$ (as defined in \cite{BeCo2}).  
To proceed further, one will need to develop techniques to obtain more precise estimates for $C_\Sigma(E)$  in the general case.

\section{Multipliers}\label{Multipliers}

We will let $MA(G)$ (resp. $M_0A(G)$) denote the space of multipliers (resp. cb-multipliers) on $G$, as defined  for example in \cite{DCHa, Pis}, and considered in  \cite{BeCo2}.
We recall that $$B(G) \subset UB(G) \subset M_0A(G) \subset MA(G)$$ 
where $B(G)$ denotes the Fourier-Stieltjes algebra of $G$
and $UB(G)$ denotes the algebra consisting of the
matrix coefficients of  uniformly bounded representations of $G$. (It is also well known that all these spaces coincide whenever $G$ is amenable.)
Our aim in this section is to introduce some similar spaces for $\Sigma$. 

 \medskip In \cite{BeCo3} we introduced the notion of so-called {rf-multipliers} of $\Sigma$, giving rise to certain bounded maps from the reduced to the full crossed product. 
A related concept is as follows.

\medskip 
Let $T: G\times A \to A$ be a map which is linear in the second variable. For each $g$, we let $T_g: A \to A$ be the linear map obtained by setting 
$$T_g(a) = T(g,a)\, \quad a \in A\,.$$
Moreover, for each $f \in C_c(\Sigma)$, we define $T\cdot f \in C_c(\Sigma)$ by $$\big(T\cdot f\big)(g) = T_g(f(g))\, , \quad g\in G\,.$$
Then we say that $T$ is a (reduced) {\it multiplier} of $\Sigma$ whenever there exists a (necessarily unique) bounded linear map
$M_T: (C_r^*(\Sigma), \|\cdot\|) \to (C_r^*(\Sigma), \|\cdot\|)$ satisfying
$$ M_T\, \Lambda_\Sigma\, (f) = \Lambda_\Sigma(T\cdot f)\,, $$
that is, 
$$M_T\Big( \sum_{g \in G}\, f(g)\, \lambda_\Sigma(g)\Big) = \sum_{g \in G}\, T_g(f(g))\, \lambda_\Sigma(g)\, ,$$
for all $ f \in C_c(\Sigma)$.
Note that $M_T$ is then uniquely determined by 
$$M_T\big(a \, \lambda_\Sigma(g)\big) = T_g(a) \, \lambda_\Sigma(g)\, , \quad a \in A, \, g \in G\,.$$
Hence we have $\|T_g(a) \| \leq \|M_T\| \|a\|$ for every  $g \in G$ and $ a\in A$, and it follows that 
each $T_g$ is bounded with $\|T_g\|\leq \|M_T\|$. Especially, 
 the family $\{T_g\}_{g \in G}$ is necessarily (uniformly) bounded. 

We also note that for each $x \in C_r^*(\Sigma)$ we have
\begin{equation}\label{mult-Fourier1}
\widehat{M_T(x)}(g) = T_g\big(\widehat{x}(g)\big)\, , \quad g\in G\,.
\end{equation} 
Indeed, this is easily seen to be true when $x\in \Lambda_\Sigma(C_c(\Sigma))$, and the assertion then follows from a density argument. Conversely, if there exists a bounded linear map $M_T$ from $C_r^*(\Sigma)$ into itself such that (\ref{mult-Fourier1}) holds for every $x$ in $C_r^*(\Sigma)$, then it follows readily that $T$ is a multiplier of $\Sigma$. Thus this might be taken as an alternative definition. 

\bigskip 
We will let $MA(\Sigma)$ denote the set of all (reduced) multipliers on $\Sigma$. 
This set, that always contains the trivial multiplier $I_\Sigma$ (defined by $I_\Sigma(g,a)=a$ for all $g\in G, a\in A$),
has an obvious vector space structure, and can be equipped with the norm given by $$\tn T \tn=\|M_T\|\, .$$ We will also consider its subspace $M_0A(\Sigma)$ consisting of (reduced) {\it cb-multipliers}, that is, multipliers on $\Sigma$ satisfying that $M_T$ is completely bounded \cite{Pau, Pis}, i.e. $\|M_T\|_{cb} < \infty$.

\bigskip The simplest  conceivable kind of multipliers on $\Sigma$
are those arising from  scalar-valued functions on the group. Consider $\varphi : G\to \Complessi$ and  let $T^\varphi: G\times A \to A$ be defined by $$T^\varphi(g, a) = \varphi(g) \, a \,  , \quad a \in A, \, g \in G\,.$$
It is not difficult to see that if $\, T^\varphi \in MA(\Sigma)$ (resp. $T^\varphi \in M_0A(\Sigma)$), then $\varphi \in MA(G)$ (resp. $\varphi \in M_0A(G)$).
It is not clear to us that $\, T^\varphi \in MA(\Sigma)$ whenever $\varphi \in MA(G)$. 
However, as we will soon deduce from a more general result, we do have $T^\varphi \in M_0A(\Sigma)$ whenever $\varphi \in M_0A(G)$. 
As a warm-up, we will first show that
  $T^\varphi \in M_0A(\Sigma)$ whenever 
  $\varphi \in UB(G)$.

\medskip The following lemma,  in the vein of Fell's classical absorption principle, is similar to \cite[Lemma 2.1]{DCHa}.

\begin{lemma}\label{fell-classical}
Let $v$ be a  uniformly bounded representation  of $G$ in the group of bounded invertible linear operators on some Hilbert space $\K$. 

Then there exists an invertible adjointable operator
$V$ on the Hilbert $A$-module $A^\Sigma \otimes \K$ satisfying
\begin{equation} \label{fell-eq}
V (a \lambda_\Sigma(g) \otimes I_\K) V^{-1}
= a \lambda_\Sigma(g) \otimes v(g) \, , \quad a \in A, \, g \in G\,. 
\end{equation}
\end{lemma}

\begin{proof}  
Let $W$ be the  invertible
adjointable operator on
$A^G \otimes \K \cong (A \otimes \K)^G$ given by
$$(W \zeta)(g) = ({\rm id} \otimes v(g)) \zeta(g) , 
\ \zeta \in (A \otimes \K)^G \,, \,  g \in G \ . $$
 Then, for every vector of the form $b \otimes \delta_h \otimes \eta\in 
A^G \otimes \K \cong A \otimes \ell^2(G) \otimes \K$,
we have
\begin{align*}
W (\tilde\ell(a) \tilde{\lambda}_\ell(g) \otimes I_\K)
(b \otimes \delta_h \otimes \eta) & 
= W \big( \alpha^{-1}_{gh}(a\, \sigma(g,h)) b \otimes \delta_{gh} \otimes \eta\big) \\
& =  \alpha^{-1}_{gh}(a \, \sigma(g,h))b \otimes \delta_{gh} \otimes v(gh)\eta \\
& =  \alpha^{-1}_{gh}(a \,\sigma(g,h))b \otimes \delta_{gh} \otimes v(g)v(h)\eta \\
& = \big(\tilde\ell(a) \tilde{\lambda}_\ell(g) \otimes v(g)\big) ( b \otimes \delta_h \otimes v(h)\eta) \\
& = \big(\tilde\ell(a) \tilde{\lambda}_\ell(g) \otimes v(g)\big) W (b \otimes \delta_h \otimes \eta) \, ,
\end{align*}
 By a density argument, we get
 \begin{equation*} 
W (\tilde\ell(a) \tilde{\lambda}_\ell(g) \otimes I_\K) W^{-1}
= \tilde\ell(a) \tilde{\lambda}_\ell(g) \otimes v(g) \, , \quad a \in A, \, g \in G\,. 
\end{equation*}
We can now define $V$ on $ A^\Sigma \otimes \K $ by $V= (J\otimes I_\K) W (J^*\otimes I_\K)$
and the identity (\ref{fell-eq}) readily follows.

\end{proof}

Note that when $v$ is a unitary representation of $G$ on $\K$, the operator $V$ in Lemma \ref{fell-classical} is unitary. Especially, choosing $v = \lambda$ (the left regular representation of $G$ on $\ell^2(G)$), we get an injective homomorphism $\delta_\Sigma: C_r^*(\Sigma) \to C_r^*(\Sigma) \otimes C_r^*(G)$, called the {\it (reduced) dual coaction} of $G$ on $C^*_r(\Sigma)$, by setting 
$$\delta_\Sigma (x) = V(x\otimes I) V^*\, , \quad x \in C_r^*(\Sigma)\,. $$
The reader may for instance consult the appendix of \cite{EKQR} for a survey of the vast area of coactions on C$^*$-algebras and their crossed products. We won't need this theory in this paper, but we will make a couple of remarks involving $\delta_\Sigma$ in Section 6.

\bigskip The next proposition generalizes  \cite[Theorem 2.2]{DCHa}. 

\begin{proposition}\label{mpd}
Let $v$ be a  uniformly bounded representation  of $G$ into the bounded invertible operators on some Hilbert space $\K$ and set $K = \sup \big\{ \| v(g) \| \, ; \, g \in G\big\} < \infty$.  

\noindent Let $\eta_{1}, \eta_{2} \in \K$ and define $\varphi \in UB(G)$ by 
$$\varphi(g) = \big\langle \eta_{1} \, , v(g)\, \eta_{2}\big\rangle \, , \quad  g \in G\, .$$

\noindent Then $T^\varphi \in M_0A(\Sigma)$. Moreover, setting $M_\varphi= M_{T^\varphi}: C^*_{r}(\Sigma)\to C^*_{r}(\Sigma)$, we have
$$M_\varphi \big(a\lambda_\Sigma(g)\big) = \varphi(g) \, a\, \lambda_\Sigma(g)\, , \quad a \in A\, , \, g \in G\, ,$$
and 
$$\tn T^\varphi \tn= \| M_\varphi \| \leq \|M_\varphi\|_{cb} \leq K^2 \,  \|\eta_{1}\| \|\eta_{2}\|\,. $$

\end{proposition}

\begin{proof}
  Let $V$ be the operator  on 
$A^\Sigma \otimes \K$ obtained in Lemma \ref{fell-classical}.

\smallskip \noindent For $\eta \in \K$,  let  
 $\theta_{\eta} : A^\Sigma \to A^\Sigma \otimes \K$ be the adjointable operator given by
$$\theta_{\eta}(\xi) = \xi \otimes \eta \ , \xi \in A^\Sigma\, ,$$
its adjoint being determined by  $\theta_{\eta}^*\, (\xi' \otimes \eta') = \langle \eta,\eta' \rangle \,  \xi', \, \xi' \in A^\Sigma, \, \eta' \in \K$.

\medskip Now define $M_{\varphi}: \L(A^\Sigma)\to \L(A^\Sigma)$ by
$$M_{\varphi}(x)=\theta_{\eta_{1}}^* \, V (x \otimes I_{\K})V^{-1}\theta_{\eta_{2}} \, .$$
Then $M_{\varphi}$ is completely bounded (see \cite{Pau}),
with $$\|M_\varphi\|_{cb}\leq \|\theta_{\eta_{1}}^* \, V\|\,  \| V^{-1}\theta_{\eta_{2}}\|\leq 
K^2 \, \|\eta_{1}\| \, \|\eta_{2}\|\, .$$
Let $\xi \in \A^\Sigma$. Using Lemma \ref{fell-classical}, we get
\begin{align*}
M_{\varphi}\big(a \lambda_\Sigma(g) \big) \, \xi 
& =  \theta_{\eta_{1}}^* V \big(a\lambda_\Sigma(g) \otimes I_\K\big) V^{-1} \theta_{\eta_{2}} \,
\xi \\
& =  \theta_{\eta_{1}}^* (a \lambda_\Sigma(g) \otimes v(g)) \theta_{\eta_{2}} \, \xi \\
& = \theta_{\eta_{1}}^* \big(a \lambda_\Sigma(g) \xi \otimes v(g) \eta_{2}\big) \\
& = \big\langle \eta_{1},v(g)\eta_{2}\big\rangle \, a \lambda_\Sigma(g)\,  \xi \\
& = \varphi(g) \,  a \lambda_\Sigma(g)\,  \xi 
\end{align*}
This shows that $T^\varphi \in M_0A(\Sigma)$, with $M_{T^\varphi} = M_\varphi$, and the final assertion follows from this.

\end{proof}

When $\sigma$ is trivial, the following corollary is due to Haagerup (see \cite[Lemma 3.5]{Ha}).

\begin{corollary}\label{posdef}
Let $\varphi$ 
be a positive definite function on 
$G$. 

\smallskip \noindent Then $T^\varphi \in M_0A(\Sigma)$, $ M_{T^\varphi}: C^*_{r}(\Sigma)\to C^*_{r}(\Sigma)$ is 
completely positive,
and $$\tn T^\varphi \tn =\|M_{T^\varphi}\| = \|M_{T^\varphi}\|_{cb} =  \varphi(e)\,.$$
\end{corollary}

\begin{proof}
Since we can write $\varphi$  
as $\varphi(g) = \langle \eta \, , v(g)\, \eta \rangle \,,  g \in G\, ,$ for a suitable unitary representation $v$ of $G$ 
on some Hilbert space $\H$, the result follows readily from Proposition \ref{mpd} and its proof.

\end{proof}

\bigskip More generally, given a map $\varphi: G\to A$, we may consider the maps $L^\varphi$ and $R^\varphi$ from  $G\times A$ into $A$ given by $$L^{\varphi}(g, a) = \varphi(g) \, a\, , \quad \quad \,R^{\varphi}(g, a) = a \, \varphi(g)\, ,   \quad g \in G, \, a \in  A \, .$$ 
Inspired by the known characterizations of $M_0A(G)$ (see e.g. \cite{BF1, BF2, Jo, Pis}), we will give in Theorem  \ref{CB-mult} some conditions ensuring that these maps belong to $M_0A(\Sigma)$. For a different result about $MA(\Sigma)$, see  Proposition \ref{CK}.

\medskip We will 
use the following lemma:

\begin{lemma} \label{Linf} Let $X$ be a Hilbert $A$-module and $\zeta \in \ell^\infty(G,X)$. Define $V_\zeta: A^G \to X^G$ by $$\big[V_\zeta\, \xi\big](g) =  \zeta(g) \cdot \xi(g)\,, \quad \xi \in A^G, \, g \in G\,.$$
Then $V_\zeta \in \L(A^G, X^G), \,  \|V_\zeta\| \leq \|\zeta\|_\infty$ and $V_\zeta^{\,*}$ is given by $V_\zeta^{\, *} = W_\zeta:X^G\to A^G$, where
$$\big[W_\zeta\,\eta\big](g) = \big\langle \zeta(g), \eta(g)\big\rangle\, , \quad \eta \in X^G, \, g \in G\, .$$

\end{lemma}

\begin{proof}  Let $\xi \in A^G$ and define $\gamma: G \to X$ by 
$$\gamma(g) =  \zeta(g) \cdot \xi(g)\,, \quad 
g \in G\,.$$
Then, for each $g \in G$, we have $$ \big\langle \gamma(g), \gamma(g)\big\rangle = \xi(g)^* \big\langle \zeta(g), \zeta(g)\big\rangle\,  \xi(g) \leq  \| \zeta(g) \|^2 \, \xi(g)^* \xi(g) \leq \| \zeta \|_\infty^2 \,  \xi(g)^* \xi(g)  \,.$$
It follows  that $\gamma \in X^G$ and 
$$\|\gamma\|^2 = \big\| \, \sum_{g \in G} \big\langle \gamma(g), \gamma(g)\big\rangle\, \big\| \leq  \| \zeta \|_\infty^2 \,\,  \big\| \sum_{g\in G} \xi(g)^* \xi(g)\, \big\| \, .$$ This shows that $V_\zeta$ is well-defined and 
$\|V_\zeta\, \xi \| \leq  \| \zeta \|_\infty \, \| \xi\|$.

\medskip Similarly, let $\eta \in X^G$ and define $\delta: G \to A$ by
$$\delta(g) = \big\langle \zeta(g), \eta(g)\big\rangle\, , \quad \, g \in G\, .$$
Then, for each $g \in G$, we have
$$\delta(g)^*\delta(g) = \big\langle \zeta(g), \eta(g)\big\rangle^*\, \big\langle \zeta(g), \eta(g)\big\rangle \leq \|\zeta(g)\|^2 \, \big\langle\eta(g), \eta(g)\big\rangle \leq \|\zeta\|_\infty^2 \, \big\langle\eta(g), \eta(g)\big\rangle\,.$$
It follows  that $\delta \in A^G$. This shows that $W_\zeta: X^G\to A^G$ is well-defined. 

\medskip Now, let $\xi \in A^G, \eta \in X^G$. Then we have
\begin{align*}
\big\langle V_\zeta \,\xi, \eta\big\rangle & = \sum_{g \in G} \big\langle (V_\zeta\xi)(g), \eta(g)\big\rangle =\sum_{g \in G} \big\langle \zeta(g) \cdot \xi(g), \eta(g)\big\rangle \\
& =\sum_{g \in G}  \xi(g)^* \big\langle \zeta(g), \eta(g)\big\rangle =\sum_{g \in G}  \xi(g)^* (W_\zeta \eta)(g) = \big\langle \xi, W_\zeta \,\eta\big\rangle\, .
\end{align*}
Hence, $V_\zeta$ is adjointable with $V_\zeta^{\, *} = W_\zeta$, as desired.

\end{proof}

\begin{theorem}\label{CB-mult}
Let $\varphi : G \to A$ be a map and $L^\varphi, \, R^\varphi: G\times A \to A$ be defined as above.  Let $\pi$ be a representation of $A$ on some Hilbert $A$-module $X$ and $\eta_1, \eta_2 \in \ell^\infty(G, X)$.

\medskip 
\noindent l) Assume that
\begin{equation}\label{l1}
\pi(a)\, \eta_2(t)= \eta_2(t) \cdot a\, \quad\quad \quad  \text{for all} \, \, a\in A, \, t \in G\, ,
\end{equation}
\begin{equation}\label{l2}
\varphi(st^{-1}) = \alpha_s\big(\big\langle \eta_1(s), \eta_2(t)\big\rangle\big)\, \quad \text{for all} \, \, s, t \in G\, .
\end{equation}

\medskip \noindent Then $L^\varphi \in M_0A(\Sigma)$, and $\, \tn L^\varphi\tn \leq \|M_{L^\varphi}\|_{cb} \leq  \|\eta_1\|_\infty \, \|\eta_2\|_\infty$. 

\medskip \noindent If $\eta_1 = \eta_2$, then $M_{L^\varphi}$ is completely positive and $\tn L^\varphi\tn = \|M_{L^\varphi}\|_{cb} = \|\varphi(e)\|$. 
  
\bigskip \noindent  r) Assume that
\begin{equation}\label{r1}
\pi(a)\, \eta_1(t)= \eta_1(t) \cdot a\, \quad \quad \quad \quad \quad  \text{for all} \, \, a\in A, \, t \in G\, ,
\end{equation}
\begin{equation}\label{r2}
\varphi(st^{-1}) = {\rm Ad} (\sigma(s, st^{-1})) \, \alpha_s\big(\big\langle \eta_1(s), \eta_2(t)\big\rangle\big)\, \quad \text{for all} \, \, s, t \in G\, .
\end{equation}

\medskip \noindent Then $R^\varphi \in M_0A(\Sigma)$, and $\, \tn R^\varphi\tn \leq \|M_{R^\varphi}\|_{cb} \leq  \|\eta_1\|_\infty \, \|\eta_2\|_\infty$. 

\medskip \noindent If $\eta_1 = \eta_2$, then $M_{R^\varphi}$ is completely positive and $\tn R^\varphi\tn = \|M_{R^\varphi}\|_{cb} = \|\varphi(e)\|$.

\end{theorem}

\begin{proof}  
We use the notation of Lemma \ref{Linf} and set $V_j=V_{\eta_j} \in \L(A^G, X^G), \, j=1,2$. 

\medskip \noindent We may then define a completely bounded linear map $M: \L(X^G) \to \L(A^G)$  by $$M(S) = V_1^* S\, V_2 \, , \quad S \in \L(X^G)\, ,$$
and Lemma \ref{Linf} gives $\|M\|_{cb} \leq \|V_1\| \, \|V_2\| \leq \|\eta_1\|_\infty \, \|\eta_2\|_\infty\, .$ Clearly,  $M$ is completely positive if $\eta_1 = \eta_2$.

\medskip \noindent Consider $a \in A, \, g \in G, \, \xi \in A^G$. Then, for each $h \in G$, we have
$$ \big[M\big( \tilde{\pi}(a) \, \tilde{\lambda}_\pi(g)\big)\, \xi\big](h)= \big[V_1^* \, \tilde{\pi}(a) \, \tilde{\lambda}_\pi(g)  \,V_2\, \xi\big](h) =  \Big\langle \eta_1(h)\, , \,  \big(\tilde{\pi}(a) \, \tilde{\lambda}_\pi(g)V_2 \,\xi\big)(h) \Big\rangle\, .
$$
We also note that
\begin{align*}
\big(\tilde{\pi}(a) \, \tilde{\lambda}_\pi(g)V_2\, \xi\big)(h) & = \pi\big( \alpha_h^{-1}(a) \alpha_h^{-1}(\sigma(g, g^{-1}h)\big) (V_2\, \xi)(g^{-1}h) \\
& = \pi\big( \alpha_h^{-1}(a\, \sigma(g, g^{-1}h))\big) 
\big(\eta_2(g^{-1}h) \cdot \xi(g^{-1}h)\big)\, \\
& =  \Big(\pi\big( \alpha_h^{-1}(a\, \sigma(g, g^{-1}h))\big)\, 
\eta_2(g^{-1}h)\Big)\cdot \xi(g^{-1}h)
\end{align*}
where we have used 
$A$-linearity at the last step. Hence, we get
\begin{equation}\label{eq} \big[M\big( \tilde{\pi}(a) \, \tilde{\lambda}_\pi(g)\big)\, \xi\big](h) =  \Big\langle \eta_1(h)\, , \,
\pi\big( \alpha_h^{-1}(a\, \sigma(g, g^{-1}h))\big) 
\, \eta_2(g^{-1}h)\Big\rangle \, \xi(g^{-1}h)\,.
\end{equation}
We divide the rest of the proof in two cases.

\medskip \noindent $i)$ Assume that the assumptions in $l)$ are satisfied. Assumption (\ref{l1}) gives
$$\pi\big( \alpha_h^{-1}(a\, \sigma(g, g^{-1}h))\big) 
\, \eta_2(g^{-1}h)
= \eta_2(g^{-1}h) \cdot \alpha_h^{-1}\big(a\, \sigma(g, g^{-1}h)\big)\,,$$
 so we get
\begin{align*}
 \big[M\big( \tilde{\pi}(a) \, \tilde{\lambda}_\pi(g)\big)\,\xi\big](h) & =
  \Big\langle \eta_1(h)\, , \, \eta_2(g^{-1}h) \cdot \alpha_h^{-1}\big(a\, \sigma(g, g^{-1}h)\big)\Big\rangle\,   \xi(g^{-1}h)  
 \\
 & =   \Big\langle \eta_1(h)\, , \, \eta_2(g^{-1}h) \Big\rangle \, \, \alpha_h^{-1}\big(a\, \sigma(g, g^{-1}h)\big)\, \xi(g^{-1}h)\, .
\end{align*}

\medskip \noindent As assumption (\ref{l2}) gives $\,\,  \varphi(g) = \varphi\big(h (g^{-1}h)^{-1}\big) = \alpha_h\big(\langle \eta_1(h), \eta_2(g^{-1}h)\rangle\big)$,\\ we have 
$$\alpha_h^{-1}(\varphi(g)) = \big\langle \eta_1(h), \eta_2(g^{-1}h)\big\rangle\, .$$
Hence, we get
\begin{align*}
\big[M \big(\tilde{\pi}(a) \, \tilde{\lambda}_\pi(g)\big)\, \xi\big](h) & =
\alpha_h^{-1}(\varphi(g))\, \, \alpha_h^{-1}\big(a\, \sigma(g, g^{-1}h)\big)\, \xi(g^{-1}h) \\
& =\alpha_h^{-1}\big(\varphi(g) \, a\, \sigma(g, g^{-1}h)\big)\, \xi(g^{-1}h) \\ 
& = \big[\, \tilde{\ell}\big(\varphi(g) \, a\big) \, \tilde{\lambda}_\ell(g)\,  \xi\, \big] (h)\,. 
\end{align*}

\medskip  \noindent By linearity, we get $\, M\,(\tilde{\pi} \times \tilde{\lambda}_\pi)(f) = \Lambda\big(L^\varphi\cdot f\big)\, $ for all $f \in C_c(\Sigma)$. As $M$ is bounded, it follows that $M$ maps
$(\tilde{\pi} \times \tilde{\lambda}_\pi)\big(C^*(\Sigma))$ into $\Lambda(C^*(\Sigma))\subset \L(A^G)$.

\medskip \noindent As $\tilde{\pi} \times \tilde{\lambda}_\pi$ is weakly contained in $\Lambda$ (cf.\ \cite[p. 188]{BeCo3}), and $\Lambda_\Sigma$ is unitarily equivalent to $\Lambda$ (via $J$), there exists a homomorphism $\psi: C_r^*(\Sigma) \to \L(X^G)$ such that $\psi \, \Lambda_\Sigma = \tilde{\pi} \times \tilde{\lambda}_\pi$. 

\medskip Let now $\tilde{M}: C_r^*(\Sigma) \to C_r^*(\Sigma)$ be given by
$\tilde{M} = ({\rm Ad}\, J)\, M\, \psi$. Since $M$ is completely bounded, $\tilde{M}$ is also completely bounded. Moreover, we have
$$\tilde{M}\,\Lambda_\Sigma(f) = ({\rm Ad}\, J)\, M\, \psi \,\Lambda_\Sigma(f) = ({\rm Ad}\, J)\, M\, (\tilde{\pi} \times \tilde{\lambda}_\pi)(f)$$
$$= ({\rm Ad}\, J)\,\Lambda\big(L^\varphi\cdot f\big) = \Lambda_\Sigma \big(L^\varphi\cdot f\big)$$
for all $f \in C_c(\Sigma)$. This shows that $L^\varphi \in M_0A(\Sigma)$ with $M_{L^\varphi} = \tilde{M}$. Moreover,  we get $$\tn L^\varphi\tn = \|M_{L^\varphi}\| \leq \|\tilde{M}\|_{cb} \leq \|M\|_{cb} \leq \|\eta_1\|_\infty \, \|\eta_2\|_\infty\,.$$
If $\eta_1=\eta_2$, then $M_{L^\varphi}= \tilde{M}$ is a composition of completely positive maps and therefore itself completely positive. As $M_{L^\varphi}(1) = \varphi(e)$, the last assertion of $l)$ readily follows. 

\medskip \noindent
$ii)$ Assume now that the assumptions in $r)$ are satisfied. Equation (\ref{eq}) gives
$$ \big[M\big( \tilde{\pi}(a) \, \tilde{\lambda}_\pi(g)\big)\, \xi\big](h) =  \Big\langle \pi\big( \alpha_h^{-1}(a\, \sigma(g, g^{-1}h))\big)^* \eta_1(h)\, , \,
  \eta_2(g^{-1}h)\Big\rangle \, \xi(g^{-1}h)\,.$$
Using assumption (\ref{r1}), it follows that
\begin{align*}
\big[M\big( \tilde{\pi}(a) \, \tilde{\lambda}_\pi(g)\big)\, \xi\big](h) 
& =  \Big\langle   \eta_1(h)\cdot\alpha_h^{-1}(a\, \sigma(g, g^{-1}h))^*\, , \,
  \eta_2(g^{-1}h)\Big\rangle \, \xi(g^{-1}h)\, \\
& = \alpha_h^{-1}(a\, \sigma(g, g^{-1}h))\, \Big\langle   \eta_1(h)\, , \,
  \eta_2(g^{-1}h)\Big\rangle \, \xi(g^{-1}h)\, .
\end{align*}
  Now, assumption (\ref{r2}) gives 
  $$\varphi(g) = \varphi\big(h(gh^{-1})^{-1}\big) = \sigma(g, g^{-1}h) \, \alpha_h\big(\langle \eta_1(h), \eta_2(g^{-1}h)\rangle\big)\, \sigma(g, g^{-1}h)^*\,,$$
  hence
  $$ \big\langle \eta_1(h), \eta_2(g^{-1}h)\big\rangle= \alpha_h^{-1}\big(\sigma(g, g^{-1}h)^*\, \varphi(g) \sigma(g, g^{-1}h)\big)\, ,$$
  and we obtain
$$ \big[M\big( \tilde{\pi}(a) \, \tilde{\lambda}_\pi(g)\big)\, \xi\big](h) =
\alpha_h^{-1}(a\, \sigma(g, g^{-1}h))\,   \alpha_h^{-1}\big(\sigma(g, g^{-1}h)^*\, \varphi(g) \sigma(g, g^{-1}h)\big)\, \xi(g^{-1}h)\, $$
$$= \alpha_h^{-1}\big(a\, \varphi(g)\big)\, \alpha_h^{-1}\big( \sigma(g, g^{-1}h)\big)\, \xi(g^{-1}h)\, 
 = \big[\, \tilde{\ell}\big(a\, \varphi(g) \big) \, \tilde{\lambda}_\ell(g)\,  \xi\, \big] (h)\,. 
$$

\medskip  \noindent By linearity, we get $\, M\,(\tilde{\pi} \times \tilde{\lambda}_\pi)(f) = \Lambda\big(R^\varphi\cdot f\big)\, $ for all $f \in C_c(\Sigma)$. Clearly, we can now proceed as in the previous case to finish the proof.

\end{proof}

\medskip 
\begin{example} \label{equiv-funct}
 Let $(\rho,v)$ be an equivariant representation of 
$\Sigma$ on a Hilbert module $X$ and $x,\, y \in X$.  Assume that $x$ or $y$  lies in the central part $Z_X$ 
of $X$. Let $\varphi:G\to A$ be given by 
$$\varphi(g) = \big\langle x,  v(g)\, y \big\rangle\, , \quad g \in G\, , $$ 

\smallskip \noindent and consider the associated map $L^\varphi$ if $y\in Z_X$ (resp. $R^\varphi$ if $x \in Z_X$). 

\medskip Then an elementary computation, using that $Z_X$ is left invariant by each $v(g)$ (see \cite{BeCo3}), gives that $\varphi$ satisfies the assumptions in part $l)$ (resp.\ part $r)$) of Theorem \ref{CB-mult} with $$\xi_1(s) = {\rm ad}_\rho\big(\sigma(s, s^{-1})\big)\, v(s^{-1})\, x\, , \quad \xi_2(t)= v(t^{-1})\,y\, , \quad \text{ when} \,\,  \, y \in Z_X \, $$ 
(resp. 
$$\xi_1(s) = v(s^{-1})\,x\, ,\quad   \xi_2(t)=  {\rm ad}_\rho\big(\sigma(t^{-1}, t)^*\big)\, v(s^{-1})\,x\, ,   \quad \text{ when} \, \, \, x \in Z_X\, ).$$ 

Hence, we get  that  $L^\varphi$ (resp. $R^\varphi$) is a cb-multiplier of $\Sigma$ when $y \in Z_X$ (resp. $x \in Z_X)$. We will in fact give an alternative approach  in Example \ref{funct-ex}.

\hfill$\square$
\end{example}

Theorem \ref{CB-mult} is well known when $A =\Complessi$ and $\sigma = 1$. Indeed, consider $\varphi: G \to \Complessi$. As shown in  \cite[Theorems 5.1 and 6.4]{Pis} (see also \cite{Jo}, \cite{BF1}), $\varphi \in M_0A(G)$ if and only if there
there exist a Hilbert space $\K$ and $\xi_1, \xi_2 \in \ell^\infty(G, \K)$ such that
\begin{equation}\label{gilbert} \varphi(st^{-1}) = \big\langle \xi_1(s), \xi_2(t)\big\rangle\, \quad \text{for all} \, \, s, t \in G\,.
\end{equation}
Moreover, in this case, letting $M_\varphi: C_r^*(G) \to C_r^*(G)$ denote the associated completely bounded map, we have $\|M_\varphi\|_{cb} = \inf \, \|\xi_1\|_\infty \,  \|\xi_2\|_\infty$, where the infimum is taken over all possible pairs $\xi_1, \xi_2$ satisfying the above conditions for some Hilbert space $\K$.

\medskip Using this characterization of $M_0A(G)$, we get:

\begin{corollary} \label{cb-mult} Let $\varphi \in M_0A(G)$. Then $T^\varphi \in M_0A(\Sigma)$ and $\tn T^\varphi\tn \leq \|M_{T^\varphi}\|_{cb} \leq \|M_{\varphi}\|_{cb}$. 

\end{corollary}

\begin{proof} Define $\varphi_A:G\to A$ by $\varphi_A(s) = \varphi(s)\cdot 1$. Then pick  a Hilbert space $\K$ and  $\xi_1, \xi_2 \in \ell^\infty(G, \K)$ such that (\ref{gilbert}) holds. 

\medskip Consider the Hilbert $A$-module $X=A\otimes\K$ and the canonical representation $\pi$ of $A$ on $X$ (determined by $\pi(a)(b\otimes\zeta) = ab \otimes \zeta$).
Moreover, define $\eta_1, \eta_2 \in \ell^\infty(G, X)$ by $\eta_j(s) = 1 \otimes \xi_j(s)\, ,\,j=1,2$.

\medskip\noindent Then, trivially, $\eta_2$ satisfies the assumption (\ref{l1}) in part $l)$  of Theorem \ref{CB-mult}. Further,  for all $s, t \in G$, we have 
$$\alpha_s\big(\langle(\eta_1(s), \eta_2(t)\rangle\big)= \alpha_s\big(\langle 1, 1\rangle\, \langle \xi_1(s), \xi_2(t)\rangle\big) = \langle \xi_1(s), \xi_2(t)\rangle\cdot 1 = \varphi_A(st^{-1})\, .
$$
Hence, the assumption (\ref{l2}) in $l)$ is also satisfied (with $\varphi_A$)  and we may apply Theorem \ref{CB-mult}. We get $\, T^\varphi = L^{\varphi_A} \in M_0A(\Sigma)$, and
$$\tn T^\varphi \tn  \leq \|M_{T^\varphi}\|_{cb} =  \|M_{L^{\varphi_A}}\|_{cb} \leq \|\eta_1\|_\infty \, \|\eta_2\|_\infty = \|\xi_1\|_\infty \, \|\xi_2\|_\infty\, .$$ 
As this holds for any choice of $\xi_1, \xi_2$ satisfying (\ref{gilbert}), we get $\tn T^\varphi \tn  \leq \|M_{T^\varphi}\|_{cb} \leq \|M_\varphi \|_{cb}$, as asserted.

\end{proof}

We will now show how one may produce cb-multipliers of $\Sigma$ associated with equivariant representations of $\Sigma$, in a more general way than the one outlined in Example \ref{equiv-funct}. 
When $A$ is trivial, the basic ingredient in an equivariant representation consists of a unitary representation of the group on some Hilbert space, and the associated multipliers are then just given by Proposition \ref{mpd}. 
In the general case, our procedure is technically much more involved  and requires some preparations. We first state the result.

\begin{theorem}\label{Phi-thm} 
Let  $(\rho,v)$ be an equivariant 
representation of $\Sigma$ on a Hilbert $A$-module $X$ and let $x, y \in X$.  Define $T:G\times A \to A$ by
$$T(g,a) = \big\langle x\, , \, \rho(a) \, v(g) \, y \big \rangle \, , 
\quad g\in G, \, a\in A\, .$$

\noindent Then $T \in M_0A(\Sigma)$ and $$\tn T\tn = \| M_T \| \leq \| M_T \|_{\rm cb}  \leq \|x\| \, \|y\|\,.$$
Hence,  
$
M_T: C_r^*(\Sigma) \to C_r^*(\Sigma)
$
 satisfies 
\begin{equation}\label{MT}
M_T \big( a\, \lambda_{\Sigma}(g)\big) 
= \big\langle x\, ,\, {\rho}(a)\, {v}(g)\, y \big\rangle\, \lambda_\Sigma(g)
\end{equation}
for each $\, a \in A, \, g \in G\,.$ 

\medskip \noindent If $x=y$, then $M_T$ is completely positive and $\tn T\tn = \| M_T \| = \| M_T \|_{\rm cb}= \|x\|^2$.
\end{theorem}

We define $B(\Sigma) $ to be the set of all multipliers of $\Sigma$ obtained as in Theorem \ref{Phi-thm}, thinking of it as the set of  $A$-valued matrix coefficients associated with equivariant representations of $\Sigma$.
 Theorem \ref{Phi-thm} then says that $$B(\Sigma) \subset M_0A(\Sigma)\,.$$ 
Note that 
$I_\Sigma \in B(\Sigma)$ (taking $X=A, (\rho, v)= (\ell, \alpha), \xi=\eta=1$).
Moreover, $M_0A(\Sigma)$ (and $MA(\Sigma)$) can be endowed with an algebra structure,  $B(\Sigma)$ and $UB(G)$ may be seen as unital subalgebras of $M_0A(\Sigma)$, while $B(G)$ may be seen as a unital subalgebra of both $B(\Sigma)$
and $UB(G)$. 
To do this, we would especially need to discuss the notion of tensor product of equivariant representations of $\Sigma$. As this is somewhat lengthy and would take us away from our main focus in this paper, we will not elaborate on this any further here.

\medskip Our proof of Theorem \ref{Phi-thm} will rely on a new version of 
the Fell's absorption principle. It uses the machinery developed in \cite[Section 4]{BeCo3}, where  another
analogue of Fell's principle was established \cite[Theorem 4.11]{BeCo3}.  

\medskip We start with a lemma. Note that if $X$ is a Hilbert $A$-module, $\ell_\Sigma$ being a representation of $A$ on the Hilbert $A$-module $A^\Sigma$, 
 we may form
 the internal tensor product $X \otimes_{\ell_\Sigma} A^\Sigma$, which is also a Hilbert $A$-module (see \cite[Chapter 4]{La1}).

\begin{lemma} \label{W} Let  $(\rho,v)$ be an equivariant 
representation of $\Sigma$ on a Hilbert $A$-module $X$.

\medskip \noindent There exists a unitary operator  
$W \in \L(X \otimes_{\ell_\Sigma} A^\Sigma\, , \, X^G)$ which satisfies
\begin{equation}\label{eqW}
\big[W\, (x \dot\otimes \xi)\big] (g) = v(g)^{-1}\, \big(x \cdot \xi(g)\big)
 \end{equation}
for all $\  x \in X\, , \, \xi \in A^\Sigma \,, \, g \in G\,.$ 
\end{lemma} 
\begin{proof} 

We first define $W$ on the dense subspace $Y$ of $X \otimes_{\ell_\Sigma} A^\Sigma$ consisting of the span of elements of the form $x \dot\otimes \xi\,,$ where $ \, x \in X\,, \,  \xi \in C_c(\Sigma)$. 
For $y = \sum_{i=1}^{n} x_i \dot\otimes \xi_i \in Y$, we set
$$(Wy)(g)= v(g)^{-1} \Big(\sum_{i=1}^{n}  x_i \cdot \xi_i(g)\Big)\, , \quad g \in G\,.$$
Then we have
\begin{align*}
\big\langle Wy\, , \, Wy\big\rangle & = 
\sum_{g\in G}   \, \sum_{i,j=1}^n \Big\langle v(g)^{-1}\big(x_i \cdot \xi_i(g)\big)\, , \, v(g)^{-1}\big(x_j \cdot\xi_j(g)\big)\Big\rangle \\
& = \sum_{g\in G}  \, \sum_{i,j=1}^n {\alpha_g}^{-1}\big(\big\langle x_i \cdot \xi_i(g)\, , \, x_j \cdot\xi_j(g)\big\rangle\big) \\
& = \sum_{g\in G}  \, \sum_{i,j=1}^n {\alpha_g}^{-1}\big( \xi_i(g)^*\, \big\langle x_i \, , \, x_j \big\rangle \, \xi_j(g)\big) \\
& =  \sum_{g\in G} \,  \sum_{i,j=1}^n {\alpha_g}^{-1}\big( \xi_i(g)^*\, \big(\ell_\Sigma\big(\big\langle x_i \, , \, x_j \big\rangle \, \xi_j\big)(g)\big) \\
& =   \sum_{i,j=1}^n  \,\sum_{g\in G}  {\alpha_g}^{-1}\big( \xi_i(g)^*\, \big(\ell_\Sigma\big(\big\langle x_i \, , \, x_j \big\rangle \, \xi_j\big)(g)\big) \\
& =  \sum_{i,j=1}^n  \big\langle \xi_i\, , \, \ell_\Sigma\big(\langle x_i \, , \, x_j \rangle \big) \, \xi_j\big\rangle
= \sum_{i,j=1}^n  \big\langle x_i \dot\otimes \xi_i \, , \, x_j \dot\otimes \xi_j \big\rangle  = \big\langle y\, , \, y\big\rangle 
\end{align*}
as interchanging the sums is allowed, the $\xi_i$'s being assumed to have finite support. 

It follows that $W$ is a well defined isometry from $Y$ into $X^G$, that satisfies equation (\ref{eqW}) by definition. It extends to an isometry, also denoted by $W$, from $X \otimes_{\ell_\Sigma} A^\Sigma$ into $X^G$. As the range of $W$ obviously contains $C_c(G,X)$, $W$ is surjective.  

Moreover, $W$ is $A$-linear: It clearly suffices to check that $W \big((x \dot\otimes \xi)\cdot a\big)=\big(W (x \dot\otimes \xi)\big) \cdot a$  for all $\, x \in X, \,\xi \in A^\Sigma, \,a\in A$. Now, for every $g\in G$, we have
\begin{align*}
\big[W \big((x \dot\otimes \xi)\cdot a\big)\big] (g) & =\big[W \big((x \dot\otimes (\xi \times a)\big)\big] (g) = v(g)^{-1}\, \big(x \cdot (\xi(g)\alpha_g(a))\big) \\
& = v(g)^{-1}\, \big((x \cdot \xi(g))\cdot \alpha_g(a)\big) =  \big(v(g)^{-1}\,(x \cdot \xi(g))\big)\cdot a \\
& =\big(\big[W (x \dot\otimes \xi)\big] (g)\big) \cdot a =  \big[\big(W (x \dot\otimes \xi)\big) \cdot a\big] (g)\,,
\end{align*}
where we 
have used property $(iv)$ of equivariant representations. This shows our assertion. 

Thus, $W$ is a bijective, $A$-linear isometry and it follows from \cite[Theorem 3.5]{La1} that $W$ is unitary. 

\end{proof}

Here is our new version of Fell's classical absorption principle.

\begin{theorem} \label{fell} 
Let $(\rho,v)$ be an equivariant representation of $\Sigma$ on a Hilbert 
$A$-module $X$.

\medskip \noindent Let $(\ell_\Sigma)_{*}:\L(X) \to \L(X \otimes_{\ell_\Sigma} A^\Sigma)$ 
 denote the  canonical homomorphism associated with $\ell_\Sigma$, 
 
 \smallskip \noindent so 
$\rho\dot\otimes\ell_\Sigma = ({\ell_\Sigma})_{*}\circ \rho : 
A  \to \L(X \otimes_{\ell_\Sigma} A^\Sigma)$.

\smallskip \noindent Then the product covariant representation 
$(\rho \dot\otimes \ell_\Sigma\, , \, v \dot\otimes \lambda_\Sigma)$ 
of $\Sigma$ on $X \otimes_{\ell_\Sigma} A^\Sigma$
is unitarily equivalent to
the regular covariant representation
$(\tilde\rho\, ,\, \tilde\lambda_\rho)$ of $\Sigma$ on $X^G$. Hence, we have 
$$\, (\rho \dot\otimes \ell_\Sigma) \times (v \dot\otimes \lambda_\Sigma)\, \simeq \, \tilde\rho\times\tilde\lambda_\rho\,.$$
 
\end{theorem}

\begin{proof}
Let $W \in \L(X \otimes_{\ell_\Sigma} A^\Sigma\, , \, X^G)$ be the unitary defined in Lemma 
\ref{W}.
To prove the assertion, it is enough to prove the equalities 
$\tilde\rho(a)W = W (\rho \dot\otimes \ell_\Sigma)(a)$
and 
$\tilde\lambda_\rho(g)W = W(v \dot\otimes \lambda_\Sigma)(g)$ for all  $a\in A, \, g\in G$.
We will check these on a total set of vectors.

\smallskip We have
\begin{align*} 
\big[\tilde\rho(a)W (x \dot\otimes \xi)\big](h) 
& = \rho(\alpha_h^{-1}(a))\big([W(x \dot\otimes \xi)](h)\big) \\ 
& = \rho(\alpha_h^{-1}(a))\big([v(g)^{-1}(x \cdot \xi(h))]\big) \\ 
& = v(h)^{-1} \rho(a)(x \cdot \xi(h)) \\ 
& = v(h)^{-1} \big((\rho(a)x) \cdot \xi(h)\big) \\ 
& = \big[W(\rho(a)x \dot\otimes \xi)\big] (h) \\ 
& = \big[W((\ell_\Sigma)_* \circ \rho)(a)(x \dot\otimes \xi)\big] (h) \\ 
& = \big[W\big(\rho \dot \otimes \ell_\Sigma\big)(a)(x \dot\otimes \xi)\big] 
(h) \ ,
\end{align*}
for all $a \in A, x \in X, \xi \in A^\Sigma, h \in G$.

\medskip Similarly, using property $(ii)$ of equivariant representations, we have
$$ v(g) = {\rm ad}_\rho(\sigma(g, g^{-1}h)) \, v(h) v(g^{-1}h)^{-1}\, , \quad g,h \in G\,,$$
and this gives
\begin{align*} 
\big[\tilde\lambda_\rho(g)W (x \dot\otimes \xi)\big](h) 
& = \rho(\alpha_h^{-1}(\sigma(g,g^{-1}h))) 
\big([W(x \dot\otimes \xi)](g^{-1}h)\big) \\ 
& = v(h)^{-1} \rho(\sigma(g,g^{-1}h)) v(h)\,  v(g^{-1}h)^{-1} 
(x \cdot \xi(g^{-1}h)) \\
& = v(h)^{-1} \Big( \big(v(g) (x \cdot (\xi(g^{-1}h)) \big) \cdot 
\sigma(g,g^{-1}h) \Big) \\
& = v(h)^{-1} \big((v(g)x) \cdot (\lambda_\Sigma(g)\xi)(h)\big) \\ 
& = \big[W(v(g)x \dot\otimes \lambda_\Sigma(g)\xi)\big] (h) \\
& = \big[W\big(v \dot\otimes\lambda_\Sigma\big)(g)
(x \dot\otimes\xi)\big](h) \ ,
\end{align*}
for all $x \in X, \xi \in A^\Sigma, \, g, h \in G$.

\end{proof}

To prove Theorem \ref{Phi-thm}, we will use the same notation as in the proof of Theorem \ref{fell}.

\bigskip \noindent {\it Proof of Theorem \ref{Phi-thm}}.  As discussed in \cite[Section 2]{BeCo3},
the Hilbert $A$-modules $X^G$ and $A^G \otimes_\rho X$ are unitarily 
equivalent via the map $U \in \L (A^G \otimes_\rho X, X^G)$ given by
$$
\big[ U (f \dot\otimes x)\big] (h) = \rho(f(h))x\,, 
\quad f \in A^G, \,x \in X, \,h \in G\, .
$$
Moreover, letting $\rho_*: \L(A^G)\to \L(A^G\otimes_\rho X)$ denote the canonical homomorphism, we have
$$U \rho_*(\tilde\ell(a)) U^* = \tilde\rho(a) \, \quad \quad\, \,\,   \text{for all}\, \, a\in A\, ,$$
$$U \rho_*(\tilde\lambda_\ell(g)) U^* = \tilde\lambda_\rho(g) \, \quad \quad \text{for all}\, \, g\in G\, . 
\ $$

\medskip For $x \in X$, let 
$\theta_x \in \L(A^\Sigma,X \otimes_{\ell_\Sigma} A^\Sigma)$ 
be defined as in \cite[Lemma 4.6]{La1}, that is,
$$\theta_x(\xi) = x \dot\otimes \xi\, ,\quad   \xi \in A^\Sigma\,.$$
Then, for all $y \in X, \, \eta \in A^\Sigma$, we have  $$\theta_x^*(y \dot\otimes \eta) = \ell_\Sigma(\langle x, y 
\rangle) \, \eta = \langle x, y 
\rangle \, \eta$$
(since we identify $A$ with $\ell_\Sigma(A)$). 

\medskip Let $x, y \in X$ be given. Then define a  linear
 map $\Phi : \L(A^G) \to \L(A^\Sigma)$   by
\begin{equation*}
\Phi(\cdot)=\theta_x^* \, W^* U \, \rho_{*}(\cdot)\, U^* W
\,\theta_y\, .
\end{equation*}
Then $\Phi$ is completely bounded (see e.g.\ \cite{Pau}), with 
$$\|\Phi\|\leq \|\Phi\|_{cb} \leq 
\|x\| \, \|y\|\, .$$
Moreover, if $x=y$, then 
$\Phi(\cdot)=\theta_x^* \, W^* U \, \rho_{*}(\cdot)\, U^* W 
\theta_x$ becomes completely positive and satisfies
$$\|\Phi\|= \|\Phi\|_{cb}= \|\Phi(I)\| = \|x\|^2\,.$$

Now, for all $a \in A, g \in G, \, \xi \in A^\Sigma$,  we compute
\begin{align*}
\Phi\big(\tilde\ell(a)\tilde\lambda_\ell(g)\big) \xi
& = \theta_x^*\, W^* \, U\,  \rho_*\big(\tilde\ell(a)\tilde\lambda_\ell(g)\big)
\, U^*\,  W\,  \theta_y \, \xi \\
& = \theta_x^*\,  W^*\,  \tilde\rho(a) \tilde\lambda_\rho(a)\,  W\,  (y \dot\otimes \xi)
\\
& = \theta_x^* \,  (\rho \dot\otimes \ell_\Sigma)(a) \, 
(v \dot\otimes\lambda_\Sigma)(g) \, (y \dot\otimes \xi) \\
& = \theta_x^*\,  \big[(\ell_\Sigma)_*\big(\rho(a)\big)\big] \, \big(v(g)y \dot\otimes 
\lambda_\Sigma(g)\,  \xi\big) \\
& = \theta_x^* \big( \rho(a) v(g)y \dot\otimes \lambda_\Sigma(g)\xi\big) 
\\
& = \big\langle x, \rho(a)v(g)y \big\rangle\,  \lambda_\Sigma(g)\, \xi \,.
\end{align*}
Hence, letting $M_T$ be the restriction of  $\Phi \circ {\rm Ad}(J^*)$ to $C_r^*(\Sigma) \subseteq \L(A^\Sigma)$,
we clearly get a completely bounded map 
$M_T : C_r^*(\Sigma)  \to C_r^*(\Sigma) $ satisfying $$M_T(a\lambda_\Sigma(g)) =  \big\langle x, \rho(a)v(g)y \big\rangle\,  \lambda_\Sigma(g) = T_g(a) \,  \lambda_\Sigma(g)$$
for all $\,a\in A, \, g \in G$. This means that $T \in M_0A(\Sigma)$,  and $M_T$ satisfies the desired properties.  

\hfill $\square$

\medskip 
\begin{example} \label{funct-ex} (Example \ref{equiv-funct} revisited).
 Let $(\rho,v)$ be an equivariant representation of 
$\Sigma$ on a Hilbert module $X$ and $x,\, y \in X$. Let $T$ denote the associated multiplier of $\Sigma$, as in Theorem \ref{Phi-thm}. Assume that $x$ or $y$ lies in the central part $Z_X$ 
of $X$ and let $\varphi: G \to A$ be given by $\varphi(g) = \big\langle x,  v(g)y \big\rangle$.
Then, as  $Z_X$ is left invariant by each $v(g)$ (cf. \cite{BeCo3}),  $T$ is given 
by $$T(g,a) = \big\langle x,  v(g)y \big\rangle\, a = \varphi(g) \, a \, \quad \text{if} \, \, y\in Z_X, $$ or as
$$T(g,a) = a \, \big\langle x ,  v(g)y \big\rangle\, = a\, \varphi(g) \,  \quad \text{if} \, \, x\in Z_X .$$
Hence, we recover the multipliers considered in Example \ref{equiv-funct}.
Note that if $x$ and $y$ both lie in $Z_X$, then $\varphi$ takes its values in $Z(A)$ (the center of $A$). 
Moreover, if $x=y \in Z_X$, then we have
$\varphi(g) = \big\langle x,  v(g)x \big\rangle$ for all $g \in G$, so $\varphi$ is of positive type (w.r.t. $\alpha$) in the sense of Anantharaman-Delaroche \cite{AD1} (assuming $\sigma$ is trivial). We do not know whether functions from $G$ to $A$ of  positive type give rise to  multipliers of $\Sigma$ in general.

\hfill $\square$
\end{example}

\begin{example} Let $w$ be a unitary representation of $G$ on a 
Hilbert space $\H$ and let $(\rho,v)$ be an equivariant representation of 
$\Sigma$ on the Hilbert module $X$. One can then consider the Hilbert 
$A$-module $X \otimes \H$ 
and the equivariant representation 
$(\rho \otimes \iota, v \otimes w)$ of $\Sigma$ on $X \otimes \H$. (We leave to the reader to verify that this is indeed an equivariant representation;
note that $(\rho \otimes \iota, v \otimes w)$ will not necessarily give an equivariant representation  if $w$ is assumed to be a uniformly bounded representation of $G$).

\medskip Let $x,y \in X$, $\xi,\eta \in \H$, 
so that $x \otimes \xi, y \otimes \eta \in X \otimes \H$. Then,
by Theorem \ref{Phi-thm}, we get a multiplier $T' \in M_0A(\Sigma)$ given by
\begin{align*}
T'(g,a) & = \big\langle x \otimes \xi\, ,  (\rho \otimes \iota)(a)(v \otimes 
w)(g) (y \otimes \eta)\big\rangle  \\
& = \big\langle x \otimes \xi\, , \, \rho(a) v(g)y \otimes 
w(g)\eta\big\rangle \\
& = \big\langle x\, ,\,  \rho(a) v(g)y \big\rangle \, 
\big\langle \xi\,, w(g)\eta \big\rangle \, , \quad a\in A, \, g\in G\,.
\end{align*}

Note that if $(\rho,v)=(\ell,\alpha)$ and $x=y=1$,  then  $T'(g,a) = \langle \xi, w(g)\eta \rangle \, a$, so $T'= T^\varphi$ where $\varphi(g) = \langle \xi, w(g)\eta \rangle$. Thus $B(G)$ naturally embeds into $B(\Sigma)$.

\smallskip This example is an illustration that we have $B(G) \, B(\Sigma) \subseteq B(\Sigma)\,$ (with respect to the natural product structure in $B(\Sigma)$).

\hfill$\square$

\end{example}

\begin{example} \label{reg-equiv} Let $(\rho,v)$ be an equivariant representation of 
$\Sigma$ on $X$ and consider the induced regular equivariant 
representation $(\check\rho,\check{v})$ of $\Sigma$ on $X^G$. 
Theorem \ref{Phi-thm} gives that the map 
$(g,a) \to \langle \xi, \check\rho(a)\check{v}(g)\,\eta \rangle$ is a cb-multiplier of $\Sigma$ for any $\xi,\eta \in X^G$. 

Note that this fact can also be deduced from  
\cite[Proposition 4.13]{BeCo3}
(by letting $(\pi,u)$ in this proposition be $(\ell_\Sigma,\lambda_\Sigma)$
and using that $\tilde\rho \times \tilde\lambda_\rho$ is weakly contained in $\Lambda_\Sigma$).

\hfill $\square$
\end{example}

\begin{example}  Let $\beta$ be an endomorphism of $A$ and assume $\beta$ satisfies the following two conditions: 
\begin{itemize} 
\item[i)] $\, \beta \alpha_g = \alpha_g \beta\, $ for all  $g \in G$.
\item[ii)] $\, \beta\big(\sigma(g,h)\big) = \sigma(g,h)\, $ for all  $g ,\,  h \in G$.
\end{itemize}
Then one checks easily that $\beta$ extends to an endomorphism $\tilde{\beta}$ of $\Lambda_\Sigma\big(C_c(\Sigma)\big)$ satisfying 
\begin{equation} \label{beta}
\tilde{\beta}(a\, \lambda_\Sigma(g)) = \beta(a) \, \lambda_\Sigma(g)\, , \quad a\in A, \, g\in G\,.
\end{equation}
To show that  $\beta$ extends to an endomorphism of $C_r^*(\Sigma)$, we can consider the equivariant representation of $\Sigma$ on $A$ given by $(\rho_\beta, \alpha)$, where $\rho_\beta(a)\, b = \beta(a) \, b, \, a, b \in A$. We leave it as an exercise to verify that conditions $(i)$ and $(ii)$ imply that this is indeed an equivariant representation. Choosing $x=y=1 \in A$ gives  $\langle x, \rho_\beta(a) \alpha(g) y\rangle = \beta(a)\, $ for all $a \in A, \, g\in G$, so Theorem \ref{Phi-thm} tells us that there exists a cb-map $T_\beta$ on $C_r^*(\Sigma)$ satisfying equation (\ref{beta}). Since $T_\beta$ coincides with $\tilde{\beta}$ on $\Lambda_\Sigma\big(C_c(\Sigma)\big)$, it follows that $T_\beta$ is an endomorphism of $C_r^*(\Sigma)$ extending $\beta$, as desired.

\hfill $\square$
\end{example}

One of our motivations for studying multipliers is that they naturally appear in the context of summation processes for Fourier series of elements in $C^*_r(\Sigma)$, that we will discuss in the next section.  As in \cite{BeCo2}, which deals with the case where $A=\Complessi$, we will be interested in multipliers that have some kind of smoothing property. 

To explain this, consider    $T \in MA(\Sigma)$ and $x \in C_r^*(\Sigma)$. Recall that we have

\begin{equation}\label{mult-Fourier}
\widehat{M_T(x)}(g) = T_g\big(\widehat{x}(g)\big)\, , \quad g\in G\,.
\end{equation} 
This means that the Fourier series of $M_T(x)$ is $$\sum_{g\in G} \, T_g\big(\widehat{x}(g)\big)\, \lambda_\Sigma(g)\,.$$ 
In general, there is no reason why this series should converge w.r.t.\ the operator norm  for all $x$ in $C_r^*(\Sigma)$, i.e.\,  it may happen that $M_T(x) \not\in CF(\Sigma)$ for some $x \in  C_r^*(\Sigma)$.
 We therefore define 
$$ MCF(\Sigma)=\big\{\,T\in MA(\Sigma) \mid M_T(x) \in CF(\Sigma)\, \, \textup{for all} \, \, x \in C_r^*(\Sigma)\big\}\,.$$

Following the proof of \cite[Proposition 4.7]{BeCo2}, one can check that  $MCF(\Sigma)$ 
consists of all maps $T: G \times A \to A$ that are linear in the second variable and satisfy that  the series $\sum_{g\in G} \, T_g(\widehat{x}(g))\, \lambda_\Sigma(g)$ converges  w.r.t. $\|\cdot\|$
for every $x\in C_r^*(\Sigma)$. 

\medskip Of course, if $T\in MA(\Sigma)$ has {\it finite $G$-support}, that is, $T_g= 0$ for all but finitely many $g$'s in $G$, then the Fourier series of $M_T(x)$ is just a finite sum for every  $x\in C_r^*(\Sigma)$, so $T \in MCF(\Sigma)$. But one can easily find examples whitout finite $G$-support:

\bigskip 
\begin{example} \label{ell1} Let $\varphi \in \ell^1(G)$. As $\ell^1(G)\subset \ell^2(G) \subset B(G)$,  $T^\varphi \in M_0A(\Sigma)$. 
Moreover, $$\widehat{M_{T^{\varphi}}(x)}=  \widehat{M_\varphi(x)} = \varphi \, \widehat{x} \in \ell^1(G, A)$$ for all $x \in C_r^*(\Sigma)$, so $M_{T^\varphi}(x) \in CF(\Sigma)$ for all $x$, hence $ T^\varphi \in MCF(\Sigma)$. 

\smallskip \noindent When $A=\Complessi$, it is not difficult to show that $T^\varphi \in MCF(\Sigma)$ whenever $\varphi \in \ell^2(G)$, cf.\ \cite[Section 4, p.\, 356]{BeCo2}). But the argument given there does not carry over to the general case, and we do not know if this assertion always holds when $A$ is non-trivial.

\hfill $\square$
\end{example}

\medskip 

The next  proposition shows how  multipliers belonging to $MCF(\Sigma)$ may be produced in a way similar to \cite[Lemma 1.7]{Haa1} and \cite[Proposition 4.8]{BeCo2}. It explains why the $A^\Sigma_{\kappa}$-spaces introduced in Section 3 have to be taken into consideration.

\begin{proposition}\label{CK}
Let $\kappa: G \to [1, \infty)$ and assume  that $\Sigma$ has the 
$A^\Sigma_{\kappa}$-decay property with decay constant $C$.

\medskip \noindent Let $\psi \in \ell^\infty_\kappa(G,A)$, that is, $\psi:G\to A$ satisfies  $K= \|\psi \kappa\|_\infty= \sup_{g \in G} \|\psi(g)\kappa(g)\| < \infty$.

\medskip \noindent Let $L^\psi : G \times A \to A$ be given by  $L^\psi(g,a) = \psi(g) \, a$.

\bigskip \noindent Then $L^\psi \in MCF(\Sigma)$ with $\, \tn L^\psi \tn \leq CK\,.$

\end{proposition}

\begin{proof}
Let $\xi \in A^\Sigma$. Then
\begin{align*}
\|\psi \xi\|^2_{\alpha,\kappa} & = \|\psi \xi \kappa\|^2_\alpha
=  \Big\| \sum_{g \in G}
\alpha_g^{-1}(\xi(g)^* \psi(g)^* \psi(g)\xi(g)) \;  \kappa(g)^2 \; \Big\| \\
& \leq \Big\| \sum_{g \in G}
\alpha_g^{-1}(\xi(g)^*\xi(g)) \; \|(\psi \kappa)(g)\|^2 \; \Big\| \\
& \leq K^2 \Big\| \sum_{g\in G} \alpha^{-1}_g (\xi(g)^*\xi(g))\Big\| 
= K^2 \|\xi\|^2_\alpha \ . 
\end{align*}
Hence, for any $f \in C_c(\Sigma)$, using (\ref{norm-ineq}), we  have
$$\|M_{L^\psi} \big(\Lambda_\Sigma(f)\big)\|  = \|\Lambda_\Sigma(\psi\, f)\| 
\leq C\,  \|\psi f\|_{\alpha,\kappa} \leq CK \, \| f \|_\alpha
\leq CK \, \|\Lambda_\Sigma(f)\| \ . $$
This shows that $L^\psi \in MA(\Sigma)$ with $\tn L^\psi \tn \leq CK$.

\smallskip \noindent Let $x \in C_r^*(\Sigma)$. Then $\widehat{x} \in A^\Sigma$, so  the above computation gives that $$\|\psi \,\widehat{x} \|_{\alpha,\kappa} \leq K \, \|\widehat{x} \|_{\alpha} < \infty\, .$$
 Thus $\widehat{M_{L^\psi}(x)} = \psi \, \widehat{x} \in A^\Sigma_\kappa$, and it follows from Proposition \ref{L-decay} that
 $M_{L^\psi}(x) \in CF(\Sigma)$.
 
\end{proof}

It seems very unlikely to us that Proposition \ref{CK}  remains true in general if we replace 
$A^\Sigma_{\kappa}$-decay with $\ell^2_\kappa(G,A)$-decay in the assumption.

\section{Summation processes for Fourier series}\label{Summation}

By a {\it Fourier summing net for} $\Sigma$, we will mean a net $\{T^i\}$ in $MCF(\Sigma)$ such that
\begin{equation}\label{F}
\lim _i \, \|M_{T^i}(x) - x \| = 0\,  \quad \textup{for all} \, \, x \in C_r^*(\Sigma)\,.
\end{equation} 
We will say that such a net is {\it bounded} whenever $\sup_i \tn T^i \tn < \infty$. 
We will repeatedly use the fact that, in order to show that a net  $\{T^i\}$ in $MCF(\Sigma)$ is a bounded Fourier summing net, 
one only needs to check that 
\begin{equation} \sup_i \tn T^i \tn < \infty \, \,  \, \, \text{and} \, \,\, \,  \lim_i\, T^i_g(a) = a \, \, \, \, \text{for all} \, \, g \in G, \, a \in A.
\end{equation}  
This assertion is easily shown using an $\varepsilon/3$-argument.

\medskip Assume that $\{T^i\}$ is a Fourier summing net  for $\Sigma$. Note that, as each $T^i$ is  assumed to lie in $MCF(\Sigma)$,
the series 
$$\sum_{g\in G} \, T^i_g\big(\widehat{x}(g)\big)\, \lambda_\Sigma(g)$$ is convergent in operator norm for each $ x \in C_r^*(\Sigma)$ and each $i$, and  equation (\ref{F}) gives 
$$\lim_i \, \sum_{g\in G} \, T^i_g\big(\widehat{x}(g)\big)\, \lambda_\Sigma(g) \, = x  \quad \textup{for all} \, \, x \in C_r^*(\Sigma)\,$$

\vspace{-1ex} \noindent with respect to the operator norm on $C_r^*(\Sigma)$. 

\medskip Hence, a Fourier summing net $\{T^i\}$ for $\Sigma$ provides a summation process for the Fourier series of all elements in $C_r^*(\Sigma)$.   
An interesting open question is whether there always exists a Fourier summing net for $\Sigma$. (To our knowledge, this is still open even when $A$ and $\sigma$ are trivial).

\medskip We will also be interested in Fourier summing nets satisfying an additional property: A Fourier summing net $\{T^i\}$ for $\Sigma$  will be said to
 {\it preserve the invariant ideals of $A$} if every invariant ideal of $A$ is preserved by each $T_g^i$, that is, for  every invariant ideal $J$ of $A$ we have 
 $$T_g^i(J) \subset J \, \, \text{for every}\, \, i \, \, \text{and every} \, \, g \in G\,.$$ 
 Of course, by an invariant ideal of $A$ we mean as usual an ideal of $A$ left invariant by each $\alpha_g$. 
As we will discuss below, the existence of a Fourier summing net for $\Sigma$ that preserves the invariant ideals of $A$ has some useful consequences 
when studying
the ideal structure of $C_r^*(\Sigma)$. This was first observed by Zeller-Meier when $G$ is amenable (cf. \cite[Proposition 5.10] {ZM}), and by Exel \cite{Ex} when $\Sigma$ has the approximation property (in the setting of Fell bundles). From a purely C$^*$-algebraic point of view, these nets are those of primary interest. 
However, as we will soon see, Fourier summing nets  preserving the invariant ideals  do not necessarily exist when $G$ is not {\it exact}, that is,  $C_r^*(G)$ is not exact as a C$^*$-algebra (see \cite{BrOz} and references therein).  

We recall some more terminology and introduce some notation. 

\medskip Let $J$ be an invariant ideal of $A$. The ideal of $C_r^*(\Sigma)$ generated by $J$  will be denoted by $\langle J\rangle$ and called an {\it induced ideal} of $C_r^*(\Sigma)$. Moreover,  $q: A\to \tilde{A}= A/J$ will denote the quotient map, $\tilde{\Sigma}= (\tilde{A}, G, \tilde{\alpha}, \tilde{\sigma})$ the induced quotient system (defined in the obvious way) and $\tilde{q}$  the canonical homomorphism from $C_r^*(\Sigma)$ onto $C_r^*(\tilde{\Sigma})$, determined by
$\tilde{q} \,\Lambda_\Sigma = \Lambda_{\tilde{\Sigma}}\, q\,$.
Then we set $\tilde{J}= {\rm Ker}\,  \tilde{q}$. Finally, we set
$$\check{J}= \big\{x \in C_r^*(\Sigma) \mid \widehat{x}(g) \in J \, \, \text{for all} \, \, g \in G\big\}.$$

\begin{proposition} \label{ideals}
Let $J$ be an invariant ideal of $A$. Then we have
$$\langle J \rangle \, \subset \, \tilde{J} \, \subset \check{J}\, .$$
Assume that  there exists a Fourier summing net $\{T^i\}$ for $\Sigma$  
 that preserves $J$.
 Then we have
 $$\langle J \rangle \, = \, \tilde{J} \, = \check{J}\, .$$

\end{proposition}

\begin{proof} The first inclusion is well known, at least when $\sigma$ is trivial. For completeness, we sketch the argument. Using the  invariance of $J$ and the covariance relation, 
one sees that $\langle J\rangle$ is the norm closure of $$\langle J\rangle_{alg}= \Big\{ \sum_{g \in F} a_g \, \lambda_\Sigma(g) \mid F \subset G, \,F \, \, \text{finite}, a_g \in J \,\, \text{for all}\, \, g \in F\Big\}\, .$$
As $\tilde{q}$ obviously maps $\langle J\rangle_{alg}$ to $\{0\}$ and $\tilde{J}$ is closed, it is clear that $\langle J \rangle \subset \tilde{J}$. 

\smallskip Next, it is an easy exercise to check that for all $x \in C_r^*(\Sigma)$ and $g \in G$, we have $$q(\widehat{x}(g)) = \widehat{q(x)}(g)\, .$$
It follows that if $x \in \tilde{J}$, then $q(\widehat{x}(g)) = 0$ for every $g \in G$, hence that $\widehat{x} (g)\in J$ for  every $g \in G$. This shows the second inclusion. 

\smallskip Now, assume that  there exists a Fourier summing net $\{T^i\}$ for $\Sigma$  
 that preserves $J$
 and consider $x \in \check{J}$. For every $i$, set 
 $$x_i = \sum_{g\in G} \, T^i_g(\widehat{x}(g))\, \lambda_\Sigma(g)\, .$$ 
Using the assumption, we have $T^i_g(\widehat{x}(g)) \in J$ for every $i$ and every $ g \in G$. As $\langle J\rangle$ is closed, we get $x_i \,  \in \langle J\rangle$ for every $i$. Since $x$ is the norm-limit of $\{x_i\}$, this implies that $x \in \langle J\rangle$. This shows that  $\check{J} \subset \langle J\rangle $ and the last assertion clearly  follows.

 \end{proof}

One should note that if $G$ is exact, then Exel has shown that we also have $\langle J \rangle \, = \, \tilde{J} \, = \check{J}\, $ for every invariant ideal $J$ of $A$ (see \cite[Theorem 5.2]{Ex2}\footnote{This article of Exel is the preprint version of \cite{Ex3}. It contains a section on induced ideals that was removed in the published version.}).

\smallskip We will say that $\Sigma = (A, G, \alpha, \sigma)$ is {\it exact} whenever we have $\langle J \rangle = \tilde{J}$ for every invariant ideal $J$ of $A$. When  $\sigma$ is trivial, this terminology was recently introduced  by A. Sierakowski in  \cite{Si} to give a characterization of systems $(A, G, \alpha)$ having the property that all ideals of  $C_r ^*(A, G, \alpha)$ are induced. As shown by E. Kirchberg and S. Wassermann (cf. \cite[Theorem 5.1.10]{BrOz}), it is then known that $G$ is exact if and only if $(B, G, \beta)$ is exact for every action $\beta$ of $G$ on some C$^*$-algebra $B$. In fact, $G$ is exact if and only if the system $(B, G, \beta, \omega)$ is exact for every twisted action $(\beta, \omega)$ of $G$ on a C$^*$-algebra $B$, as follows easily from \cite[Theorem 4.4]{Ex3}. Now, an immediate consequence of Proposition \ref{ideals} is:
 
 \begin{corollary} \label{exact}

Assume that  there exists a Fourier summing net $\{T^i\}$ for $\Sigma$  
 that preserves the invariant ideals of $A$.
 Then  $\Sigma$ is exact.
 
\end{corollary}
 
One may therefore wonder whether exactness of $G$ always implies the existence of a Fourier summing net for $\Sigma$ that preserves the invariant ideals of $A$. 
 
 On the other hand, assume that $G$ is not exact. This means that there exists a  C$^*$-algebra $B$ such that $(B,G,{\rm id})$ is not exact (since $C_r^*(B,G,{\rm id}) \simeq B \otimes C_r^*(G) $). Hence, it follows from Corollary \ref{exact} that there exists no Fourier summing net  for $(B,G,{\rm id})$ that preserves the (invariant) ideals of $B$.

It is also known that $G$ is exact if and only if $C_r^*(B,G, \beta, \omega)$ is exact  whenever $(\beta, \omega)$ is a twisted action of $G$ on some exact C$^*$-algebra $B$ (see \cite[Theorem 7.2 and Remark 7.4]{AD2} for the case of untwisted actions; the twisted case can be handled in a similar way). Nevertheless, if $A$ is exact and $\Sigma$ is exact, then $C_r^*(\Sigma)$ is not necessarily exact: to see this, one may for instance consider the trivial action of a non-exact group on a simple exact C$^*$-algebra. However, we have:

\begin{proposition}\label{exactcrp}
Assume that there exists a Fourier summing net $\{T^i\}$ for $\Sigma$  
 that preserves the invariant ideals of $A$. Then $C_r^*(\Sigma)$ is exact  if and only if $A$ is exact.
\end{proposition}
\begin{proof}
We only sketch the proof, as it is close to the proof of  \cite[Theorem 7.2]{AD2}. 
Assume that $A$ is exact and let $0 \to I \to B \to B/I \to 0$ be a short exact sequence for some C$^*$-algebra $B$. Consider the twisted action $(\alpha\otimes {\rm id}_B, \sigma \otimes 1_B)$ of $G$ on $A \otimes B$. 

For each $i$, let $S^i:G \times (A\otimes B) \to A \otimes B$ be given by $S_g^i= T^i_g\otimes {\rm id}_B$ for each $g \in G.$  Then it is easy to check that $\{ S^i\}$ is a Fourier summing net for 
the system $\Omega= (A\otimes B,\, G, \,\alpha\otimes {\rm id}_B,\, \sigma \otimes 1_B)$. Now $J=A \otimes I$ is an invariant ideal of $A\otimes B$ that is clearly preserved by $\{S^i\}$. Proposition \ref{ideals} gives therefore that $\langle J \rangle = \tilde{J}$. Using the obvious identification of  $C_r^*(\Omega)$ with $ C_r^*(\Sigma) \otimes B$, one then observes that this fact corresponds to the exactness of the sequence
$$ 0 \to  C_r^*(\Sigma) \otimes I  \to C_r^*(\Sigma) \otimes B \to C_r^*(\Sigma) \otimes B/I \to 0\,.$$ 
This shows that $C_r^*(\Sigma)$ is exact. The converse implication is trivial since exactness of C$^*$-algebras passes to C$^*$-subalgebras.

\end{proof}
We will show below that if there exists a Fourier summing net for $\Sigma$ that preserves the invariant ideals of $A$, then the induced ideals of $C_r^*(\Sigma)$ may be characterized by 
certain invariance properties.
 
 \medskip 
Let $\mathcal{J}$  be an ideal of $C_r^*(\Sigma)$. Then $\mathcal{J} \cap A$ is an invariant ideal of $A$, that may be equal to \{0\} even if $\mathcal{J}\neq \{0\}$. On the other hand,  $E(\mathcal{J})$ is easily seen to be an invariant algebraic ideal of $A$ that contains $\mathcal{J} \cap A$. Moreover, since $E$ is faithful, we have $E(\mathcal{J}) \neq \{0\}$ if $\mathcal{J}\neq \{0\}$. However, it is not obvious that $E(\mathcal{J})$ is necessarily closed in general.  

\medskip We will say that $\mathcal{J}$
 is $E$-{\it invariant} 
 when $E(\mathcal{J}) \subset \mathcal{J}$.
 
  \medskip Note that when $G=\Relativi$ and $\sigma$ is trivial, $E$-invariant ideals of $C_r^*(\Sigma)$ are called {\it well behaving} in \cite{Tom2}.
It is straightforward to see that an ideal $\mathcal{J}$  of $C_r^*(\Sigma)$ is $E$-invariant if and only if $E(\mathcal{J}) = \mathcal{J} \cap A$; especially, $E(\mathcal{J})$ is then a (closed) invariant ideal of $A$. 

\medskip  It is well known and easy to check that any induced ideal of $C_r^*(\Sigma)$ is $E$-invariant. It is not known in general whether the converse is true, i.e. whether any $E$-invariant ideal of $C_r^*(\Sigma)$ is induced. However, this holds whenever $G$ is exact, as shown by Exel  \cite[Corollary 5.3]{Ex2}.

\medskip  The concept of $E$-invariance is related to another kind of invariance. Following \cite{LPRS, GL}, we will say that an ideal $\mathcal{J}$ of $C_r^*(\Sigma)$ is {\it $\delta_\Sigma$-invariant} whenever $$\delta_\Sigma(\mathcal{J}) \subset \mathcal{J} \otimes C_r^*(G)\, .$$ Here $\delta_\Sigma$ denotes the (reduced) dual coaction of $G$ on $\Sigma$ defined in Section 4. It is evident that every induced ideal of $C_r^*(\Sigma)$ is $\delta_\Sigma$-invariant. Moreover, every $\delta_\Sigma$-invariant ideal  of $C_r^*(\Sigma)$ is $E$-invariant: this follows readily after checking that we have $E= ({\rm id}\otimes \tau) \, \delta_\Sigma$, where $\tau$ denotes the canonical tracial state on $C_r^*(G)$.  

Hence, if $G$ is exact, we get from Exel's result mentioned above that an ideal $\mathcal{J}$ of $C_r^*(\Sigma)$ is $E$-invariant if and only if it is $\delta_\Sigma$-invariant,  if and only if it is induced. In the case where $G$ is amenable and $\sigma$ is trivial, the last part of this statement follows from  \cite[Theorem 3.4]{GL}.

\medskip 
\begin{proposition}\label{bij}
Assume that 
$G$ is exact or that there exists a Fourier summing net $\{T^i\}$ for $\Sigma$  
 that preserves the invariant ideals of $A$.
 
 \smallskip Then an ideal of $C_r^*(\Sigma)$ is $E$-invariant if and only if it is $\delta_\Sigma$-invariant, if and only if it is induced. 
 
 Hence, the map $J \to \langle J\rangle$ is a 
 bijection between the set of all invariant ideals of $A$ and the set of all $E$-invariant ideals of $C_r^*(\Sigma)$. 
  \end{proposition}
 \begin{proof} 
  As the map $J \to \langle J \rangle$ is injective, the second assertion will follow immediately from the first.  Moreover,
 we have just seen that the first assertion holds whenever $G$ is exact. Hence, we assume that there exists a Fourier summing net $\{T^i\}$ for $\Sigma$ that preserves the invariant ideals of $A$. To show that the first assertion holds in this case, in view of our considerations above, it suffices to prove that every $E$-invariant ideal of $C_r^*(\Sigma)$ is induced. So let $\mathcal{J}$  be an $E$-invariant ideal of $C_r^*(\Sigma)$ and set $J= \mathcal{J} \cap A$, i.e. $J =E(\mathcal{J}) $.
 Note that $$\langle J \rangle \subset \mathcal{J} \subset \check{J}\, .$$
 Indeed, the first inclusion is immediate since $J$ is contained in the ideal $\mathcal{J}$. Now, let $x\in \mathcal{J}$ and $g\in G$. Then $x\,\lambda_\Sigma(g)^* \in \mathcal{J}$, so $$\widehat{x}(g) =E(x\, \lambda_\Sigma(g)^*)\in E(\mathcal{J}) = J \, .$$
 This shows that $x \in \check{J}$ and the second inclusion follows. 
Appealing to Proposition \ref{ideals}, we can then conclude that $\mathcal{J} = \langle J \rangle = \check{J}$,  hence that  $\mathcal{J} $ is induced, as desired.

 \end{proof}
 
\begin{example} \label{wPo} Assume that $G$ is a weak Powers group (see \cite{Bed, dH} and references therein), e.g. $G$ is a non-abelian free group or a  free product of non-trivial groups that is different from $\Relativi_2 * \Relativi_2$. 
We recall a few facts from \cite{Bed}. A simple $G$-averaging process on $C_r^*(\Sigma)$ is a map $\phi$ from $C_r^*(\Sigma)$ into itself such that for some $n \in \Naturali$ and $s_1, \ldots, s_n \in G$ we have
$$ \phi(x) = \frac{1}{n} \sum_{i=1}^n \lambda_\Sigma(s_i)\, x \,\lambda_\Sigma(s_i)^* \quad \text{for all} \, \, x \in C_r^*(\Sigma) \,.$$
A $G$-averaging process $\psi$ is a composition of finitely many simple $G$-averaging processes. Note that such a linear map $\psi$ is positive and maps any ideal of $C^*_r(\Sigma)$  into itself. Lemma 4.6 in \cite{Bed} says that if $x^*=x \in C^*_r(\Sigma)$  is given, then for every  $\varepsilon > 0$ there exists a $G$-averaging process $\psi_\varepsilon$ such that $\| \psi_\varepsilon\big(x-E(x)\big) \| < \varepsilon$. This lemma is used in \cite{Bed} to show that $C_r^*(\Sigma)$ is simple whenever $A$ is $\alpha$-simple, i.e. it has no other invariant ideals  than $\{0\}$ and $A$. (This result was first proved by P. de la Harpe and G. Skandalis \cite{dHSk} when $G$ is a Powers group and $\sigma$ is trivial). 

Let us now assume that
 there exists a Fourier summing net for $\Sigma$ that preserves the invariant ideals of $A$ (or that $G$ is exact).  We then know  from Proposition \ref{bij} that the invariant ideals of $A$ are in a one-to-one correspondence with the $E$-invariant ideals of $C_r^*(\Sigma)$. 
 
 We consider first the 
 case where $\alpha$ is trivial. As any $G$-averaging process then restricts to the identity map on $A$, it clearly follows from the lemma cited above that any ideal of $C_r^*(\Sigma)$ is $E$-invariant in this case. Hence,
 we get
 that the ideals of $A$ are in a one-to-one correspondence with the ideals of $C_r^*(\Sigma)$. We note that if we also assume that $A$ is commutative, then the existence of a Fourier summing net that preserves the ideals of $A$ may be deduced from Corollary \ref{BFP3} in certain cases (see Example \ref{PSL}).

When $\alpha$ is not trivial, the ideal structure of $C_r^*(\Sigma)$ can be much more complicated.
Nevertheless, we 
can  still obtain some valuable information: as we will show below, the map $J \to \langle J\rangle$ gives a bijection between the maximal invariant ideals of $A$ and the maximal ideals of $B=C_r^*(\Sigma)$.

\smallskip If $J$ is a maximal invariant ideal of $A$ (such an ideal must  exist by Zornification), then,  
using the same notation as in Proposition \ref{ideals}, $B/\langle J\rangle \simeq C_r^*(\tilde{A}, G, \tilde{\alpha}, \tilde{\sigma})$ is simple since $\tilde{A}=A/J$ is $\tilde{\alpha}$-simple (and $G$ is a weak Powers group). Hence, $\langle J \rangle$ is maximal in $B$.
   
  Next, let $\mathcal{J}$ be a {\it proper} ideal of $B$. Then $J=\overline{E(\mathcal{J})}$ is a {\it proper} ideal of $A$. 
Indeed, assume (by contradiction) that $J =A$. Then, as $A$ is unital, $E(\mathcal{J})= A$. So pick $x \in \mathcal{J}$ such that $E(x) = 1$. Then, as $E$ is a Schwarz map \cite{BrOz}, $E(x^*x) \geq E(x)^* E(x) = 1$. Thus $y = x^*x \in \mathcal{J}^+$ satisfies $E(y) \geq 1$. Using the lemma cited above, we can find a $G$-averaging process $\psi$ such that $$\| \psi(y) -\psi(E(y)) \| < \frac{1}{2}\,.$$
Since $\psi(E(y))\geq \psi(1) =1 $ and $\psi(y)$ is positive, this implies that $\psi(y) \in \mathcal{J}$ is invertible. Hence $\mathcal{J} = B$, contradicting that $\mathcal{J}$ is proper. 

Moreover, $J$ is clearly invariant and satisfies $\mathcal{J} \subset \check{J}$. Using Proposition \ref{ideals} (or the remark following it if $G$ is exact), we have $\check{J} = \langle J \rangle$, 
hence $\mathcal{J} \subset \langle J \rangle$. 

Now, if $\mathcal{J}$ is assumed to be maximal, then we get $\mathcal{J} = \langle J \rangle$ and $J$ is necessarily maximal among the invariant ideals of $A$. This proves our assertion.

\hfill$\square$ \end{example}

Following \cite{BeCo2}, we introduce some more terminology. We will say that $\Sigma$  has {\it the Fej{\' e}r  property} if there exists a Fourier summing net $\{ T^i\}$ for $\Sigma$ such that
each $T^i$ has finite $G$-support. If such a net $\{T^i\}$ can be chosen to be bounded, $\Sigma$ will be said  to have {\it the bounded Fej{\' e}r  property}. It is a well-known result due to Zeller-Meier \cite{ZM} that $\Sigma$ has the bounded Fej{\' e}r property whenever $G$ is amenable and $\sigma$ is central. (See  \cite{Da} for a short proof in the case where $G= \mathbb{Z}$ and $\sigma$ is trivial; this case is also discussed in \cite{Tom1}).  The direct analogue of Fej{\'e}r's classical summation theorem for twisted group C$^*$-algebras of amenable groups \cite[Theorem 5.6]{BeCo2} is still valid in our more general setting:

\begin{theorem}\label{Fej} Assume $G$ is amenable. Then $\Sigma$ has the bounded Fej{\' e}r  property. Indeed, pick a F\o lner net  $\{F_i\}$ for $G$  
and let $T^i : G \times A \to A$ be given by 
$$T^i (g, a) = 
\frac{|g F_i  \cap F_i|}{|F_i|} \, a \, , \quad g \in G, \, a \in A\,.$$ 
Then 
$\{T^i\}$ is a  Fourier summing net for $\Sigma$ such that each $T^i$ has finite $G$-support and $\tn T^i \tn = 1$ for each $i$.
\end{theorem}

 \begin{proof} As in \cite{BeCo2}, set 
$\varphi_i(g) = \frac{|g F_i  \cap F_i|}{|F_i|}, 
\,  g \in G.$ 
Then $\{\varphi_i\}$ is a net in $C_c(G)$ of normalized positive definite functions converging pointwise to $1$. 
As $T^i = T^{\varphi_i} $, each $T^i$ has finite $G$-support and it follows from Corollary \ref{posdef} that $T^i \in MCF(\Sigma)$ with $\tn T^i\tn = \varphi_i (e) = 1$ for each $i$. 

\end{proof}  

Theorem \ref{Fej} may be generalized to a class of groups containing nonamenable groups. We recall from \cite[Section 12.3]{BrOz} that  $G$ is called {\it weakly amenable} if there exists a net $\{\varphi_i\}$ of finitely supported functions in $M_0A(G)$ converging pointwise to $1$ which is bounded, that is, $\sup_i \|M_{\varphi_i}\|_{cb} < \infty\,,$ where $M_{\varphi_i}:C_r^*(G)\to C_r^*(G)$ denotes the completely bounded map  associated to each $\varphi_i$. The class of weakly amenable groups contains for example all amenable groups and all groups acting properly on a tree. It is closed under taking subgroups and Cartesian products. See \cite{BrOz} and references therein for other examples. 
\begin{theorem}\label{weak-a} Assume $G$ is weakly amenable. Then $\Sigma$ has the bounded Fej{\' e}r  property.
\end{theorem}

\begin{proof} This is a direct consequence of Corollary \ref{cb-mult}.
\end{proof}

When $G$ is weakly amenable, it is easy to see that the Fourier summing net for $\Sigma$ produced in the proof of Theorem \ref{weak-a} will preserve the invariant ideals of $A$, so Proposition \ref{bij} may be applied. 
Alternatively, one could use that $G$ is then known to be exact.
In fact, if $G$ is weakly amenable, then $G$ has Haagerup's and Kraus' approximation property AP \cite{HK}, and if $G$ has the AP, then $G$ is exact (see \cite[Section 12.4]{BrOz}).  
Note that there are groups having the AP without being weakly amenable \cite[p. 373]{BrOz}, and that it follows from the recent work of V. Lafforgue and M. de la Salle \cite{LS} 
 (see also \cite{HL}) that there are examples of exact groups without the AP. 
In this connection, it would be interesting to know whether 
$\Sigma$ will have the Fej{\' e}r property whenever $G$ has the AP, or more generally, whenever $G$ is exact.

\medskip Instead of conditions involving only the group $G$, one may look for conditions on $\Sigma$. In this direction, we have:

\begin{theorem} \label{BFP}
Assume that  $\Sigma$
has the weak approximation property. 
Then $\Sigma$ has the bounded Fej{\'e}r property. 

\smallskip \noindent Moreover, assume that $\Sigma$
has the approximation property, or the half-central weak approximation property. 

\smallskip \noindent Then $\Sigma$ is exact,  while $C_r^*(\Sigma)$ is exact if and only if $A$ is exact. 
We also have that the $E$-invariant ideals of $C^*_r(\Sigma)$
 are in a one-to-one correspondence with the  invariant ideals of $A$.
\end{theorem}

\begin{proof}
Let $X, (\rho, v), M, \, \{\xi_i\}$ and $\{\eta_i\}$ be as in the definition 
   of the weak approximation property, each $\xi_i$ $($resp.\ $\eta_i)$   being chosen in $X^G$ with finite support {\rm supp}$(\xi_i)$  $($resp.\ {\rm supp}$(\eta_i)$$)$.
   
 \medskip \noindent For each $i$, define $T^i: G \times A \to A$ by  
$\, T^i (g,a) =  \big\langle \xi_i\,,\,\check\rho(a)\check{v}(g)\eta_i \big\rangle\,,$ $\textup{that is,}$
\begin{equation}\label{Ti} 
T^i (g,a) = \sum_{h\in G} \big\langle \xi_i(h)\, , \,  \rho(a) \,v(g)\eta_i(g^{-1}h)\big\rangle \,, \quad g \in G, \, a\in G\, .
\end{equation}
From Theorem \ref{Phi-thm}, see also Example \ref{reg-equiv}, we know that  $T^i \in MA(\Sigma)$ and satisfies 
$\tn T^i \tn \leq  \, \|\xi_i\| \, \|\eta_i\| $ for each $i$. Since $\|\xi_i\| \, \|\eta_i\| \leq M$ for each $i$, we see that $\{ T_i\} $ is bounded.

\medskip \noindent Moreover, we have $\lim_i \| T^i(g,a)-a\| = 0$ for each $g \in G$ by assumption. Finally, equation (\ref{Ti}) gives that each $T^i$ has finite $G$-support equal to  \,{\rm  supp}$(\xi_i)\cdot \big(${\rm supp}$(\eta_i) \big)^{-1}$. Altogether, this shows that $\{T^i\}$  is a 
bounded Fourier summing net for $\Sigma$ such that each $T^i$ has finite $G$-support, and the first assertion is proven.  

\medskip Now, assume first that  $\Sigma$
has the half-central weak approximation property, which means that $\{\xi_i\}$ or $\{\eta_i\}$ may be chosen to lie in the central part of $X^G$. As shown in Example \ref{funct-ex},  each $T^i$ is then a multiplier obtained by multiplication (from the left or from the right) with a function from $G$ to $A$. It is therefore obvious that $\{T^i\}$ preserves (all) ideals of $A$.

\medskip Next, assume that  $\Sigma$
has the  approximation property, that is, we have $(\rho, v) = (\ell, \alpha)$ and $\{\xi_i\}, \{\eta_i\} \subset A^G$. Then, for every $i$ and every $ g \in G, \, a \in A$, we have 
$$T^i_g(a) = \sum_{h\in G} \xi_i(h)^* \,\, a \,\, \alpha_g\big(\eta_i(g^{-1}h)\big) \,,$$
and it is evident that $\{T^i\}$  preserves (all) ideals of $A$ also in this case.

\medskip Hence, the second part of the theorem follows from  Corollary \ref{exact}, 
 Proposition \ref{exactcrp} and Proposition \ref{bij}.

\end{proof}

Note that when $\Sigma$ has the approximation property, 
the first and the final assertions of 
Theorem \ref{BFP}  are closely related to \cite[Propositions 4.9 and 4.10]{Ex}, since  $C^*_r(\Sigma)$ may be written as the reduced sectional algebra of a Fell bundle over $G$ \cite{ExLa}. Examples of systems (with $\sigma$ trivial) satisfying a strong version of the approximation property (called amenability) may be found in \cite[Chapters 4 and 5]{BrOz} (see also \cite{AD1, AD2}). 
For such amenable systems,
the third assertion of Theorem \ref{BFP}  is already known, cf.\ \cite[Theorem 4.3.4, part (3)]{BrOz}. 

\medskip
 \begin{example}
 Assume that $G$ is exact, $H$ is an amenable subgroup of $G$,  $A=\ell^\infty(G/H)$, \\ $\alpha$ is the natural action of $G$ on $A$ and $\sigma$ takes values in $\Toro$. Then it is shown in \cite[Example 5.19]{BeCo3} that $\Sigma = (\ell^\infty(G/H), G, \alpha, \sigma)$ has the weak approximation property and one may therefore apply Theorem \ref{BFP} to produce a bounded Fej{\'e}r summing net for $\Sigma$. Moreover, it can be checked\footnote{One has then to have a closer look at the proof of \cite[Proposition 5.15]{BeCo3}: using the notation used in this proof, one checks easily that  if $\xi$  lies in the  central part of $X^G$, then $\xi'$ defined by $\xi'(g)= \xi(g) + N, \, g \in G,$ lies in the central part of $(X'_B)^G$. } that $\Sigma$ has the central approximation property and the second part of Theorem \ref{BFP} also applies. Alternatively, one could use here that $G$ is assumed to be exact.
  
 \medskip  \hfill $\square$
   \end{example}

\medskip In \cite{BeCo2}, we discussed analogs of Abel-Poisson summation of Fourier series in reduced twisted group C$^*$-algebras. In $C_r^*(\Sigma)$, the only case that is straightforward to handle is when $G=\Relativi^n$. 
Indeed, similarly to \cite[Theorem 5.7]{BeCo2}, we have:

\begin{theorem}  Let $ G=\Relativi^n$ for some $n \in \Naturali$. For $p \in \{1, 2\}$,
    let $| \cdot |_{p}$ denote the usual $p$-norm on $G$ and let $L(\cdot)$ denote either $| \cdot |_{1} \, ,\,  | \cdot |_{2}$ or 
    $| \cdot |_{2}^2\,$. 
    
 \smallskip \noindent For each $r \in (0,1)$, let $\varphi_{r} = r^{L}$ be the function on $G$ defined by $\varphi_r(g) = r^{L(g)}$ and set $T^r = T^{\varphi_r}$, so 
 $$T^r(g, a) = r^{L(g)}\, a\, , \quad g\in G, \, a\in A\,.$$
      
 \smallskip \noindent Then $\{T^{r}\}_{r \to 1^{-}}$ is a bounded Fourier summing net for $\Sigma$. \end{theorem} 
 \begin{proof} As pointed out in the proof of \cite[Theorem 5.7]{BeCo2},  $\varphi_{r}$ is a normalized 
 positive definite function on $G$ for each $r \in (0,1)$. Hence,   Corollary \ref{posdef}  gives that $\{T^r\}$ is a bounded net  in $MA(\Sigma)$. Moreover, 
 each  $\varphi_{r}$ lies in $\ell^1(G)$, so  Example \ref{ell1} gives 
 that $T^r \in MCF(\Sigma)$ for each $r \in (0,1)$.   
As 
$\varphi_{r}$  converges pointwise to 1 when $r \to 1^{-}$, we have 
$$\lim_{r \to 1^{-}}\,  \| T^r(g,a) -a \| = \lim_{r \to 1^{-}} \, | \varphi_r(g) -1| \, \|a \| =  0$$ 
for each $g \in G, \, a\in A$. Hence the result follows.

\end{proof}

To show versions of the Abel-Poisson summation theorem for systems associated with other kind of groups, such as Coxeter groups or Gromov hyperbolic groups, the following result, analogous to \cite[Proposition 5.8]{BeCo2}, might prove to be helpful (as in the case $A=\Complessi$ discussed in \cite[Section 5]{BeCo2}). We will give an application of it in the next section.
\begin{proposition} \label{BFSN}
Let $\{\psi_i\}$ be a net of functions from $G$ to $A$ converging pointwise to $1$ and
consider the maps $L^i : G \times A \to A$ given by  $L^i(g,a) = \psi_i(g) \, a$.

\smallskip \noindent Assume that 
for each $i$ there exists $\kappa_i: G \to [1, \infty)$ such  that 
\begin{itemize}
\item $\Sigma$ has the 
$A^\Sigma_{\kappa_i}$-decay property with decay constant $C_i$,  
\item $\psi_i  \in \ell^\infty_{\kappa_i}(G,A)$, so $K_i=\|\psi_i \kappa_i\|_{\infty} < \infty$.

\end{itemize}

\smallskip  \noindent Then $\{L^i\} \subset MCF(\Sigma)$.
Moreover, if $\sup_i \, C_i K_i < \infty$ or, more generally, if $\{ L^i\}$ is bounded, then 
$\{L^i\}$ is a 
bounded Fourier summing net for $\Sigma$.
\end{proposition}

\begin{proof}
According to Proposition \ref{CK}, the first two conditions ensure that $L^i \in MCF(\Sigma)$ for each $i$. Moreover, as $ \tn L^i \tn \leq  C_i K_i \, $ for each $i$
and $\lim_i \, L^i_g(a) = \lim_i\, \psi_i(g) \, a = a$ for each $g\in G$ and $a\in A$, the final assertion is clear.

\end{proof}

Assume $G$ is a Coxeter group or a Gromov hyperbolic group and let $L$ denote the algebraic length function on $G$ associated with some finite set of generators for $G$. It is known that  $\{\psi_r\}_{r\in(0,1)}$ with $\psi_r =r^L$ gives a bounded net in $M_0A(G)$ (cf. \cite{CCJJV} and \cite{Oz}). It therefore follows  from Corollary \ref{cb-mult}
that the net $L^r$ associated with $\{\psi_r\}$ (as in Proposition \ref{BFSN}) is a bounded net in $M_0A(\Sigma)$, hence in $MA(\Sigma)$. In order to  apply Proposition \ref{BFSN} and deduce that $\{L^{r}\}$ is a bounded Fourier summing net for $\Sigma$,  it suffices to show that $\Sigma$ has the $A_{\kappa_r}^\Sigma$-decay property for each $r \in (0,1)$, where $\kappa_r=r^{-L}$. Note that $G$ is $\kappa_r$-decaying (because $G$ has polynomial H-growth w.r.t.\ $L$, cf.\ \cite[Example 3.12]{BeCo3}). However we do not know if the $A_{\kappa_r}^\Sigma$-decay property may be deduced from this, except when $A$ is commutative and $\alpha$ is trivial (see Corollary \ref{L-proper}).

\medskip We also mention a result closely related to Proposition \ref{BFSN}:

\begin{proposition} \label{BFP2}
Let $\{\psi_i\}$ be a net of functions from $G$ to $A$ converging pointwise to $1$ and
consider the maps $L^i : G \times A \to A$ given by  $L^i(g,a) = \psi_i(g) \, a$.

\smallskip \noindent Assume that for each $i$ the following conditions hold: 

\begin{itemize}
\item $L^i \in MA(\Sigma)$ with $\tn L^i \tn = 1$,
\item there exists $\kappa_i: G \to [1, \infty)$ such  that 
 $\Sigma$ has the 
$A^\Sigma_{\kappa_i}$-decay property   \\ and $\psi_i \kappa_i \in c_0(G,A)$.
\end{itemize}

\smallskip  \noindent Then $\Sigma$ has the bounded Fej\'er property.
\end{proposition}

\begin{proof} The proof is  a verbatim adaptation of the proof of \cite[Theorem 7.1]{BeCo2}
(that  itself is an adaptation of \cite[Theorem 1.8]{Haa1}),
now appealing to Proposition \ref{CK} instead of invoking 
 \cite[Proposition 4.8]{BeCo2}.
 
 \end{proof}

 Note that if we stick to normalized (scalar-valued) positive definite functions $\psi_i$ on $G$ in the above assumptions, then $G$ must have the Haagerup property (see \cite{CCJJV} or \cite[Section 12.2]{BrOz}). But allowing $A$-valued functions might be useful to handle other kind of situations.  

\section{The almost trivial case}\label{Trivial}

In this final section, we take up the issue of finding examples of  weight functions $\kappa$ on $G$ such that $\Sigma$ has the $A^\Sigma_\kappa$-property in the ``almost trivial'' case where  $A=C(X)$ is commutative and $\alpha$ is trivial.
In such a  situation, $C^*_r(\Sigma)=C^*_r(C(X), G, \rm{id}, \sigma)$ is a (unital discrete) reduced central twisted transformation group algebra, and the variety of C$^*$-algebras contained in this class is larger than one might imagine at a first thought;  see for example \cite{EW} and note that any twisted reduced group C$^*$-algebra associated with  a central group extension belongs to this class.
 
\medskip  We will use the following notation. 

\smallskip \noindent For $a \in A$ and $\omega \in S(A)$ (the state space of $A$), we set $\|a\|_\omega = \omega(a^*a)^{1/2}$. 

\smallskip \noindent Let $\xi \in A^G$. For each $\omega \in S(A)$, we define $|\xi|_\omega : G\to [0, \infty)$ by $$|\xi|_\omega(g) = \|\xi(g)\|_\omega = \omega(\xi(g)^*\xi(g))^{1/2}\, ,\quad g\in G\, .$$
Note that $|\xi|_\omega \in \ell^2(G)$ since $$\| \, |\xi|_\omega \|_2^2=  \sum_{g\in G} \omega\big(\xi(g)^*\xi(g)\big)= \omega\Big(\sum_{g\in G} \xi(g)^* \xi(g)\Big) < \infty\, .$$

 \smallskip \noindent  Letting $\|\xi\|$ denote the norm of $\xi$ in $A^G$, we have
\begin{align*}
\|\xi\|= \|\sum_{g\in G} \xi(g)^* \xi(g)\|^{1/2} &= \sup_{\omega\in S(A)} \omega\big(\sum_{g \in G} (\xi(g)^* \xi(g))\big)^{1/2}
= \sup_{\omega\in S(A)} \Big(\sum_{g\in G} \omega(\xi(g)^* \xi(g))\Big)^{1/2} \\
& = \sup_{\omega\in S(A)} \Big( \sum_{g\in G} \|\xi(g)\|^2_\omega \Big)^{1/2} 
= \sup_{\omega\in S(A)} \| \, |\xi|_\omega \, \|_2\, .
\end{align*}
Similarly, if $P(A)$ denotes the pure state space of $A$, we  have
$\|\xi\|= \sup_{\omega\in P(A)} \| \, |\xi|_\omega \, \|_2\, $.

\bigskip 

\begin{lemma}\label{commineq}
Assume $A$ is commutative and $\omega$ is a pure state of 
$A$. Let $f\in C_c(\Sigma)$ 
and assume it takes values in $A^\alpha=\{ a\in A\mid \alpha_g(a) = a \, \, \text{for all}\, \,  g \in G\}$.

\smallskip \noindent Then, for all $\xi \in A^G$,
we have
$$\| \, |\Lambda(f)\xi|_\omega \, \|_2 
\leq \| \, |f|_\omega * |\xi|_\omega \, \|_2 \,.
$$
\end{lemma}

\begin{proof}

Set $E = {\rm supp}(f)$ and let $h \in G$. Note the assumption on $f$ implies that $\alpha_h^{-1}(f(g)) = f(g)$ for all $g \in E$. Moreover, note that $\|u\|_\omega = 1$ for any $u \in \U(A)$. Using the triangle inequality for $\|\cdot\|_\omega$ and the fact that $\omega$ is multiplicative, we then get
\begin{align*}
| \Lambda(f)\xi|_\omega(h) 
& = \Big\|\sum_{g \in E} \alpha_h^{-1}(f(g)) \alpha_h^{-1}(\sigma(g,g^{-1}h))\xi(g^{-1}h)\Big\|_\omega \\
& \leq \sum_{g \in E} \|f(g)\|_\omega \, \| \alpha_h^{-1}(\sigma(g,g^{-1}h)))\|_\omega \, \|\xi(g^{-1}h)\|_\omega\\
& = \sum_{g \in E} \|f(g)\|_\omega \, \|\xi(g^{-1}h)\|_\omega
= \sum_{g \in E} |f|_\omega(g) \; |\xi|_\omega(g^{-1}h) \\
& = (|f|_\omega * |\xi|_\omega) (h) \,,
\end{align*}
and the desired inequality follows immediately.

\end{proof}

It is conceivable that Lemma \ref{commineq} holds without having to assume that $f$ takes values in $A^\alpha$ if its conclusion is changed to: `` For all $\xi \in A^G$, we have $$\| \, |\Lambda(f)\xi|_\omega \, \|_2 
\leq \| \, |f^\alpha|_\omega * |\xi|_\omega \, \|_2 \,, $$ where $f^\alpha$ is defined by $f^\alpha(g) = \alpha_g^{-1}(f(g)), \, g \in G$." Sorrily, we have so far not been able to establish this inequality. With such a more general result at hand, we would not have to assume that $f$ takes values in $A^\alpha$ in the next proposition, and our results in the sequel would all also be true for a non-trivial $\alpha$.

\begin{proposition} \label{commu}
Assume $A$ is commutative and $G$ is $\kappa$-decaying with decay constant $C$ for some $\kappa: G \to [1, \infty)$.
Let $f\in C_c(\Sigma)$ and assume it takes values in $A^\alpha$.
Then 
$$\|\Lambda_\Sigma(f)\| \leq C \, \|f\|_{\alpha, \kappa}\,.$$
\end{proposition}
\begin{proof}
Let  $\xi \in A^G$. Then, using Lemma \ref{commineq} and the $\kappa$-decay of $G$, we get
\begin{align*}
\|\Lambda(f)\xi\| & 
 = \sup_{\omega\in P(A)} \| \, |\Lambda(f)\xi|_\omega\, \|_2 \\
& \leq  \sup_{\omega\in P(A)} \|\, |f|_\omega * |\xi|_\omega \, \|_2 \\
& \leq C  \sup_{\omega\in P(A)} \|\, |f|_\omega \|_{2, \kappa} \, \, \| \, |\xi|_\omega \, \|_2 \\
& =  C  \sup_{\omega\in P(A)} \|\, |f\kappa|_\omega \|_{2} \,\,  \| \, |\xi|_\omega \, \|_2 \\
& \leq  C\,    \|\, f\kappa \, \| \, \| \,\xi \, \|  = C\,   \|\, f \, \|_{\alpha, \kappa} \, \| \,\xi \, \|\, ,
\end{align*}
the final equality being  due to the fact that $f$ takes values  in $A^\alpha$.
 
\medskip \noindent This shows that $\|\Lambda(f)\| \leq  C\,   \|\, f \, \|_{\alpha, \kappa}\, $. As $\|\Lambda_\Sigma(f)\|=\|\Lambda(f)\|$, the assertion is proven.
\end{proof}
\begin{corollary} \label{comdecay} Assume $A$ is commutative,   $\alpha$ is trivial and  $G$ is $\kappa$-decaying for some $\kappa: G \to [1, \infty)$. Then $\Sigma= (A, G, \rm{id}, \sigma)$ is $A^\Sigma_\kappa$-decaying.
\end{corollary}

\begin{proof}
Since  $A^\alpha=A$ when $\alpha$ is trivial, the result is an immediate consequence of Proposition \ref{commu}. 
\end{proof}
The following result generalizes \cite[Theorem 3.13]{BeCo2}. 
\begin{corollary} \label{L-proper} Assume that $A$ is commutative,  $G$ is countable and $\alpha$ is trivial. Let $L: G \to [0,\infty)$ be a proper function.

\smallskip \noindent  Assume that $G$ has polynomial $H$-growth $($w.r.t. $L$$)$. 
Then there exists some $s_0 > 0$ such that $\Sigma$ is $A^\Sigma_\kappa$-decaying,
where $\kappa = (1+L)^{s_0}$.

\smallskip \noindent  More generally, assume that $G$ has subexponential $H$-growth $($w.r.t. $L$$)$. 
Let $r \in (0,1)$ and set  $\kappa_r = r^{-L}$.  
Then $\Sigma$ is $A^\Sigma_{\kappa_r}$-decaying. 
\end{corollary}
\begin{proof}
Assume that $G$ has polynomial $H$-growth (w.r.t. $L$). Then, according to \cite[Theorem 3.13, part 1)]{BeCo2}, there exists some $s_0 > 0$ such that $G$ is $\kappa$-decaying, where $\kappa = (1+L)^{s_0}$. So the first statement follows from Corollary \ref{comdecay}. The second statement is proven in the same way, using now \cite[Theorem 3.13, part 2)]{BeCo2}. 
\end{proof}
Assume $A$ is commutative, $G$ is countable and $\alpha$ is trivial. 
In the setting of Corollary \ref{L-proper}, the first assertion implies that if $G$ has polynomial $H$-growth (w.r.t. $L$) and we set  $\kappa_s=(1+L)^s$ for $s > 0$, then the Fr{\'e}chet  space $\cap_{s>0} A^\Sigma_{\kappa_s}$ (w.r.t. the obvious family of seminorms)  embeds (densely) in $C_r^*(\Sigma)$. We tend to believe that this also should hold when $\alpha$ is non-trivial.   

When $A$ is commutative, $\sigma$ is scalar-valued and  $\alpha$ is not assumed to be trivial, such a Fr{\' e}chet space has been considered by Ji and Schweitzer \cite{JiSc} in the more general setting of actions by locally compact groups. By following a rather different method, involving a certain generalized Roe algebra, they show that if $G$ has the so-called strong rapid decay (SRD) property (w.r.t. to some proper length function on $G$), then the associated Fr{\'e}chet space
embeds  as a {\it spectral invariant} dense $*$-subalgebra of $C_r^*(\Sigma)$. Moreover, it is shown in \cite{CW} that the converse statement is also true for discrete groups, and that, a discrete group $G$ has property (SRD) if and only if $G$ has polynomial growth in the usual sense (w.r.t. some proper length function).

\medskip We also include a result in the vein of  \cite[Theorem 5.9 and Corollary 5.15]{BeCo2}:

\begin{corollary} \label{BFP3}

Assume $A$ is commutative,  $G$ is countable with the Haagerup property and $\alpha$ is trivial. Let $L: G \to [0,\infty)$ be a Haagerup function for $G$ (so $L$ is negative definite and proper) and assume that $G$ has subexponential $H$-growth (w.r.t. $L$).

\smallskip \noindent  For each $r \in (0,1)$, set $\psi_{r} = r^{L}$ and $L^r = L^{\psi_r}$.

   \medskip \noindent Then $\{L^{r}\}_{r \to 1^{-}}$ is a bounded Fourier summing net for $\Sigma=(A, G, \rm{id}, \sigma)$. Moreover, $\Sigma$ has the bounded Fej{\'e}r property.

\end{corollary}
\begin{proof} Since each $\psi_r$ is a normalized positive definite function on $G$, we know from Corollary \ref{posdef} that $\tn L^r\tn =1$ for all $r$, so $\{ L^r\}_{r\to 1^{-}}$ is bounded. It is also evident that $\psi_r$ converges pointwise to 1 when $r\to 1^{-}$.  

For each $r \in (0,1)$, set $\kappa_r = r^{-L}$. Corollary \ref{L-proper} gives that  $\Sigma$ is $A^\Sigma_{\kappa_r}$-decaying. As $\psi_r \kappa_r = 1$, $\psi_r \in \ell^\infty_{\kappa_r}(G,A)$ 
for all $r$.  Hence, we may apply Proposition \ref{BFSN} to obtain the first assertion. 

For each $r \in (0,1)$, set $\kappa'_r = r^{-L/2} = \kappa_{\sqrt{r}}$. Then,  for each $r$, $\Sigma$ is $A^\Sigma_{\kappa'_r}$-decaying and $\psi_r \kappa'_r \in c_0(G,A)$. We may therefore apply Proposition \ref{BFP2} and obtain the second assertion.

\end{proof}
 Finally, as an application of Corollary \ref{BFP3}, we give a simple example illustrating how our work may be used to determine the ideal structure of certain group C$^*$-algebras.
 
\smallskip
 \begin{example} \label{PSL}Consider $K=SL(2,\Relativi)$ and let $\lambda$ denote its left regular representation on $\ell^2(K)$. Denote the identity element in $K$ by $I_2$ and set $S=\lambda(-I_2)$. As $S=S^*=S^{-1}$ is central in $B=C_r^*(K)$, we have $$B = Bp \oplus Bq$$
 where $p, q$ are the central projections in $B$ given by $p= \frac{I+S}{2}, \, q = \frac{I-S}{2}$.
 
 \medskip Especially, $B$ has at least two non-trivial ideals, namely $Bp$ and $Bq$. In fact, these are the only non-trivial ideals of $B$. 
To see this, we may argue as follows. 

\medskip Let $Z=\{\pm I_2\}$ denote the center of $K$ and set $G = K/Z$, so $G$ is the modular group $PSL(2, \Relativi) \simeq \Relativi_2 \ast \Relativi_3$.  Using  \cite[Theorem 2.1]{Bed}, we may write  $B \simeq C_r^*(A, G, {\rm id}, \sigma)$ where $A= C_r^*(Z) \simeq \Complessi^2$ and $\sigma:G\times G \to \U(A)\simeq \Toro^2$ is a suitably chosen coycle.
Now,  as is well known,  $G\simeq \Relativi_2 \ast \Relativi_3$ is a Powers group \cite{dH}  
and $G$ has the Haagerup property,  the ``block'' length function $L$ on $G$ being a Haagerup function (see for instance \cite{Bo}). Moreover, $G$ has polynomial H-growth, and therefore subexponential H-growth (w.r.t. $L$), see \cite[Example 3.12, part 4)]{BeCo2}. Hence,  Corollary \ref{BFP3} applies, showing the existence of a Fourier summing net for $(A, G, {\rm id}, \sigma)$, which 
obviously preserves ideals of $A$. 
We can therefore conclude from Example \ref{wPo} that there is a bijection between ideals of $A\simeq \Complessi^2$ and ideals of $B$, hence that $B$ has exactly two non-trivial ideals, as desired. 

Note that $G$ is known to be exact (cf.\ \cite[Theorem 5.2.7]{BrOz}), so one can avoid showing the existence of a Fourier summing net as we did above.
However, we hope that the technique of proof might help to handle more complicated cases in the future.
 
 \hfill $\square$
 \end{example} 

 \medskip
\noindent{\bf Acknowledgements.} 
 Most of the present work has been done during several visits
 made by R.C.  at the  Institute of Mathematics, 
University of Oslo in the period 2009--2013.
 He thanks the operator algebra group for their kind hospitality and the Norwegian Research Council for partial financial support.

\bigskip

{\parindent=0pt Addresses of the authors:\\

Erik B\'edos, Institute of Mathematics, University of
Oslo, \\
P.B. 1053 Blindern, N-0316 Oslo, Norway.\\ E-mail: bedos@math.uio.no \\

\noindent
Roberto Conti, 
Universit\`{a} Sapienza di Roma, \\
 Dipartimento di Scienze di Base e Applicate per l'Ingegneria, Sez. di Matematica,\\
 via A. Scarpa 16, I-00166 Roma, Italy.
\\ E-mail: roberto.conti@sbai.uniroma1.it\par}

\end{document}